\documentclass[11pt]{article}

\usepackage[text={153mm,230mm},centering]{geometry}
\usepackage{bm}

\usepackage{amsmath,amsthm}

\usepackage{amscd}
\usepackage{amsmath}
\usepackage{mathrsfs}
\usepackage[mathscr]{eucal}
\usepackage[all]{xy}
\usepackage{color, xcolor}
\usepackage{soul}
\usepackage{authblk}
\usepackage{tikz}

\usepackage{amsmath}
\usepackage{amssymb}
\usepackage{amsthm}
\usepackage{amsfonts}
\usepackage{mathrsfs}
\usepackage{indentfirst}
\usepackage{fancyhdr}
\usepackage{amsxtra}
\usepackage{amscd}
\usepackage{latexsym}
\usepackage{graphicx}
\usepackage{hyperref}

\newtheorem{theorem}{Theorem}[section]

\newtheorem{lemma}[theorem]{Lemma}

\newtheorem{corollary}[theorem]{Corollary}
\newtheorem{proposition}[theorem]{Proposition}
\newtheorem{claim}[theorem]{Claim}

\theoremstyle{definition}

\newtheorem{definition}[theorem]{Definition}
\newtheorem{remark}[theorem]{Remark}
\newtheorem{notation}[theorem]{Notation}

\newtheorem*{ack}{Acknowledgements}

\numberwithin{equation}{section}

\allowdisplaybreaks

\title{Tilting objects in singularity categories of\\ certain toric Gorenstein singularities
\footnotetext{Email: xjchen@scu.edu.cn, liuleilei@zust.edu.cn, zengjh662@163.com}}

\author[a]{Xiaojun Chen}

\author[b]{Leilei Liu}

\author[c]{Jieheng Zeng}

\affil[a]{School of Mathematics, Sichuan University, Chengdu, P.R. China}

\affil[b]{School of Science, Zhejiang University of Science and Technology, Hangzhou,
P.R. China}

\affil[c]{School of Mathematics, Peking University, Beijing, P.R. China}
\date{}

\begin{document}
\maketitle

\begin{abstract}
We study a class of Gorenstein isolated singularities
which are the quotients of generic and unimodular representations of
the one-dimensional torus, or of the product of the one-dimensional torus with
a finite abelian group. Based
on the works of
\v{S}penko and Van den Bergh [Invent. Math. 210 (2017), no. 1, 3-67] and Mori and Ueyama
[Adv. Math. 297 (2016), 54-92], we
show that the singularity categories of these varieties admit tilting objects, and hence are triangle
equivalent to the perfect categories of some finite dimensional algebras.


\end{abstract}

\setcounter{tocdepth}{2}
\tableofcontents

\section{Introduction}\label{sect:Intro}

For a Noetherian graded Gorenstein algebra $S$ with singularities,
its singularity category $D^{gr}_{sg}(S): = D^{b}(\mathrm{grmod}\, S)/\mathrm{Perf}(S)$
is a triangulated category, which reflects
many properties of the singularities of $S$.
For example,
Buchweitz shows in \cite{RB}
that $D^{gr}_{sg}(S)$ is triangle equivalent to the stable category
$\mathrm{\underline{CM}}^{\mathbb{Z}}(S)$,
which is the quotient category of the category of graded maximal Cohen-Macaulay modules over $S$ by
the full subcategory of graded projective modules.

In a series of papers \cite{DR, DR1, DR2},
Orlov shows that the graded singularity categories
are determined by the local properties of the singularities,
and
have a deep relationship with the Homological Mirror Symmetry conjecture, where
the category of graded D-branes of type B with a homogeneous superpotential is
equivalent to the singularity category $D^{gr}_{sg}(S)$ for some graded commutative algebra $S$.
If $S$ is a Gorenstein isolated singularity and we forget the gradings of the $S$-modules,
then $D_{sg}(S)\cong\mathrm{\underline{CM}}(S)$ is a Calabi-Yau triangulated category,
which has fruitful homological properties and applications,
and has been extensively
studied in recent years.

The purpose of this paper is to construct
a tilting object in the singularity category of some isolated Gorenstein
singularities.
For an algebraic triangulated category,
the existence of a tilting object is very important,
since it replaces the abstract triangulated category
by the concrete perfect category of modules over some algebra; see, for example,
\cite{BIY,HHK} and references therein.

Recall that every affine toric Gorenstein variety is in the form
$\mathrm{Spec}\big(\mathrm{Sym}(W)^{G}\big)$, where $W$ is a generic unimodular
representation of an abelian reductive group $G$, which is
the product of a torus and a finite abelian group (c.f. \cite[\S10.6.1]{SV1}).
When $G$ is a finite group,
Iyama and Takahashi show in \cite{IT} that the corresponding singularity category
has a tilting object. Later in \cite{MU}, Mori and Ueyama generalize the above result
to Noetherian Koszul Artin-Schelter (AS-) regular algebras.

However, for the case that $G$ contains a torus,
the singularity categories are considered very rarely in the literature.
The difficulty lies in two folds.
On the one hand, there may be infinitely many non-isomorphic
irreducible maximal Cohen-Macaulay modules
over $\mathrm{Sym}(W)^{G}$,
and we therefore cannot apply the McKay-type correspondence,
such as in \cite{IT}, to construct the corresponding quiver and hence the tilting object.
On the other hand, the non-commutative crepant resolution of the quotient singularity
is not a Koszul algebra in general,
and thus the method of \cite{MU} cannot be used to this case directly, either.
Nevertheless, we prove in this paper the following.

\begin{theorem}\label{MainTh}
Let $k$
be an algebraically closed field of characteristic zero.
Let $G$ be
\begin{enumerate}\item[$(1)$]
the one-dimensional torus $k^\times$, or
\item[$(2)$]
the product of $k^\times$ with a nontrivial finite abelian group.
\end{enumerate}
Suppose $V = k^{n}$ is a generic and unimodular representation of $G$, and
$S:=k[x_{1}, \cdots, x_{n}]^G$ is an isolated affine toric Gorenstein singularity.
Then the singularity category $D^{gr}_{sg}(S)$ admits a tilting object, where the grading of $S$
is canonically induced from that of $k[x_{1}, \cdots, x_{n}]$.
\end{theorem}

The main idea of the proof of the above theorem is as follows.
First, we use the method of \v{S}penko and Van den Bergh
\cite[Theorem 1.6.2]{SV} to construct a
non-commutative crepant resolution $\Lambda$ of $R^{G}$, where $R = k[x_{1}, \cdots, x_{n}]$.
By the results of Iyama and Reiten \cite[Lemma 3.6 \& Theorem 3.7]{IR}
and Wemyss \cite[Theorem 4.6.3]{W2},
$\Lambda$ is an $(n-1)$-Calabi-Yau algebra, and we further show that it
is an AS-regular algebra of dimension $n-1$ with Gorenstein parameter $n$ 
(see Theorem \ref{ToAS} below).
Therefore by Orlov \cite[Theorem 16]{DR2},
the bounded derived category
$D^{b}(\mathrm{tails}\, \Lambda)$ has a tilting object which is
$\bigoplus^{0}_{i = -n+1} \Lambda(i)$,
where, for a graded algebra $A$,
$\mathrm{tails}\, A := \mathrm{grmod}\, A / \mathrm{tors}\, A$
with $\mathrm{tors}\, A$ being the full subcategory of $\mathrm{grmod}\, A$
consisting of graded torsion $A$-modules.

Second, we have the following diagram which is constructed in \cite{MU}:
\begin{equation}\label{diag:first}
\begin{split}
\xymatrixcolsep{5pc}
\xymatrix{
D^{b}(\mathrm{grmod}\, \Lambda) \ar[d]_{\pi} \ar[r]^{(-)\otimes_{\Lambda}\Lambda e}
& D^{b}(\mathrm{grmod}\, R^{G}) \ar[d]_{\pi}
\ar[rd]^{\quad \nu(\textup{Verdier localization})}&\\
D^{b}(\mathrm{tails}\,\Lambda) \ar[r]^{(-)\otimes_{\Lambda}\Lambda e}
& D^{b}(\mathrm{tails}\, R^{G})
\ar@/^/@{->>}[r]^{\mu} &D^{gr}_{sg}(R^{G}),
\ar@/^/@{_{(}->}[l]^{\Phi(\textup{Orlov embedding})}
}
\end{split}
\end{equation}
where $e$ is an idempotent of $\Lambda$
(see Notation \ref{def:e} below),
$\pi$ is the natural projection,
$\mu$ is the projection from $D^{b}(\mathrm{tails}\, R^{G})$ to
its right admissible subcategory  $D^{gr}_{sg}(R^{G})$,
and $\Phi$ and $\nu$ are the Orlov embedding
and Verdier localization respectively.
Moreover, there is an isomorphism of functors
$$
\big((-)\otimes_{\Lambda}\Lambda e \big)
\circ \pi \cong \pi \circ \big((-)\otimes_{\Lambda}\Lambda e\big).
$$
We shall show that $R^{G}$ is a graded Gorenstein algebra (see Proposition \ref{ToGo} below).
By
Mori and Ueyama \cite[Lemma 2.7]{MU2},
the algebra $\Lambda/\Lambda e \Lambda$ is finite dimensional
if and only if the functor
$$
(-)\otimes_{\Lambda}\Lambda e: D^{b}(\mathrm{tails}\, \Lambda) 
\rightarrow D^{b}(\mathrm{tails}\, R^{G})
$$
is an equivalence.
The isolatedness of the singularity ensures that the tilting objects in 
$D^{b}(\mathrm{tails}\, \Lambda)$
can be transferred to $D^{b}(\mathrm{tails}\, R^{G})$ by the functor 
$(-)\otimes_{\Lambda}\Lambda e$ (see Theorem \ref{Sing} below).

Now to prove Theorem \ref{MainTh}, we improve
the method of Mori and Ueyama in \cite[Lemma 4.5]{MU}. Since
$\Lambda$ is an AS-regular algebra of dimension $n-1$ with Gorenstein parameter $n$,
we add a summand $(1 - e)\Lambda e$ to the direct sum
$\bigoplus^{n-1}_{i = 1}\big(\Omega_{\Lambda}^{i}\big((1-e)
 \Lambda_{0}\big)(i)\big)\otimes_{\Lambda}\Lambda e$
obtained in \cite{MU}, where $\Omega^i_{\Lambda}(-)$ is the $i$-th syzygy functor,
and $(1-e) \Lambda_{0}$ is considered as a graded $\Lambda$-module.

Then
taking the minimal left $\mathrm{add}(e \Lambda e)$-approximation (see Definition \ref{def:minimalapprox} below)  of
$$
\Big(\bigoplus^{n-1}_{i = 1}\big(\Omega_{\Lambda}^{i} \big((1-e) \Lambda_{0}\big)(i)\big)
\otimes_{\Lambda}\Lambda e\Big) \oplus (1 - e)\Lambda e
$$
and then applying the Verdier localization
functor $\nu(-)$, we get the tilting object
$$
\nu\left(L_{e \Lambda e}\left( \Big(\bigoplus^{n-1}_{i = 1}
\big(\Omega_{\Lambda}^{i}\big( (1-e)\Lambda_0\big)(i)\big)\otimes_\Lambda
\Lambda e \Big)\oplus (1 - e)\Lambda e\right)\right)
$$
in $D_{sg}^{gr}(R^G)$.\hfill $\square$

Finally, we remark that the method in the above proof
does not apply to the higher dimensional torus action case;
see Remark \ref{rem:highdim} for an explanation.

The rest of the paper is devoted to giving a detailed proof of the above theorem.
It is organized as follows.
In \S\ref{sect:NCCR}, we introduce some often used notations and recall some properties of
AS-regular algebras, graded Gorenstein algebras and
non-commutative crepant resolutions.
In \S\ref{sect:proofofmain1}
we prove Theorem \ref{MainTh} for the case where the group is the one dimensional torus,
and \S\ref{sect:proofofmain2} we prove the general case.
In the last section, \S\ref{sect:example},
we study an example where the tilting
object is explicitly given.

\begin{notation}
Throughout the paper, $k$ is an algebraically
closed field of characteristic zero.
All modules are right modules. The dimensions of all commutative rings and varieties are
more than one unless otherwise stated.
Also in the paper, $(-)^{\bullet}$ stands for a complex.

Suppose $A$ is a graded associative $k$-algebra. For $M \in \mathrm{grmod}\,A$
a finitely generated graded $A$-module,
we write $M = \bigoplus_{i} M_{i}$, where $M_{i}$ is the grading $i$ component.
Let $M(j)$ be the graded $A$-module such that $M(j)_{i} = M_{i+j}$.
For any $M, N \in \mathrm{grmod}\,A$, we denote
$$
\mathrm{\underline{Ext}}_{\mathrm{mod} \, A}^{i}(M, N)
: = \bigoplus_{j} \mathrm{Ext}_{\mathrm{grmod}\,A}^{i}\big(M, N(j)\big),
$$
for $i \in \mathbb{Z}$, and
$$
\mathrm{\underline{Hom}}_{\mathrm{tails} \, A}\big(\pi(M),
\pi(N)\big) : = \bigoplus_{j} \mathrm{Hom}_{\mathrm{tails}\,A}\big(\pi(M), \pi(N)(j)\big).
$$
Moreover, for a triangulated category $\mathcal{C}$
and $X, Y \in \mathrm{Ob}(\mathcal{C})$, we denote
$$
\mathrm{Hom}^{i}_{\mathcal{C}}(X, Y) : = \mathrm{Hom}_{\mathcal{C}}(X, Y[i])
\quad\textup{and}\quad
\mathrm{End}^{i}_{\mathcal{C}}(X) : = \mathrm{Hom}_{\mathcal{C}}(X, X[i])
$$
for $i \in \mathbb{Z}$.
\end{notation}

\begin{ack}
During the preparation of the paper,
the first two authors were visiting IASM, Zhejiang University.
We would like to thank Professor Yongbin Ruan as well as
Zhejiang University for hospitality.
This work is partially supported by NSFC (Nos. 12271377, 12261131498, 11890660 and 
11890663).
\end{ack}

\section{Non-commutative crepant resolutions}\label{sect:NCCR}

In this section, we study some algebraic properties of $R^G$ for $G=k^\times$.
By \v{S}penko and Van den Bergh \cite{SV},
$R^G$ has a
non-commutative crepant resolution (NCCR) in the sense of Van
den Bergh \cite{V, V2}.
We show that the NCCR thus constructed is an AS-regular algebra
of dimension $n-1$ with Gorenstein parameter $n$ (see Theorem \ref{ToAS}).

\subsection{Basics of NCCR}

We first go over the definition and
some properties of
NCCRs.
The materials are taken from \cite{SV, SV1, SV2, V, V2}.

\begin{definition}[Van den Bergh]\label{NCCR}
Let $R$ be a Gorenstein normal domain. A {\it non-commutative crepant resolution} (NCCR for short)
of $R$ is an algebra of the form $\mathrm{End}_{\mathrm{mod} \,R}(M)$ 
for some reflexive $R$-module $M$ such that
\begin{enumerate}
\item[(1)] $\mathrm{End}_{\mathrm{mod} \,R}(M)$ is a Cohen-Macaulay $R$-module, and
\item[(2)] the global dimension of $\mathrm{End}_{\mathrm{mod} \,R}(M)$ is finite.
\end{enumerate}
\end{definition}

Later Wemyss in \cite{W2} replaces condition (2) in the above definition by that
the global dimension of  $\mathrm{End}_{\mathrm{mod} \,R}(M)$ is equal to the Krull dimension of $R$. 
If $R$ is an equicodimensional Gorenstein normal domain, then both definitions coincide.

\begin{definition}\label{def:quasisymm}
Let $G$ be a reductive group, $T$ be a maximal torus of $G$ and $V$ be
a finite dimensional representation of $G$. Let
$X(T) := \mathrm{Hom}(T, k^{\times})$ be the character space of $T$.
The representation is said to be {\it quasi-symmetric}
if for every line $\ell \subset X(T)\otimes_{\mathbb{Z}}\mathbb{R}$ through the origin, 
$$
\sum_{\alpha_{i} \in \ell} \alpha_{i} = 0,
$$
where $\alpha_i$ runs through the weights in $\ell$ of $T$ on $V$.
\end{definition}

\begin{definition}\label{generi}
Let $G$ be an algebraic group and $X$ be a smooth affine variety with an action of $G$.
$G$ is said to act {\it generically} on $X$ if the action satisfies the following two conditions:
\begin{enumerate}
\item[(1)] $X$ contains a closed point with closed orbit and trivial stabilizer, and

\item[(2)] if $X^{s} \subseteq X$ is the set of points satisfying (1), then
$\mathrm{codim}(X \backslash X^{s}) \geq 2$.
\end{enumerate}
Let $V$ be a finite dimensional representation of $G$; we say the representation
is {\it generic}
if $G$ acts generically on $V$.
\end{definition}

Suppose $G$ is a reductive group that acts on finite dimensional
vector spaces $V$ and $W$. Let $R:=k[V]=k[x_{1}, x_{2}, \cdots, x_{n}]$,
with the induced action of $G$.
Denote by $M^{G}_{R}(W)$ the $R^{G}$-module $(W \otimes R)^{G}$.

Now suppose $G=T$ is a torus.
Let $(\beta_i)$ be the $T$-weights of $V$. Let
$$
\Sigma := \Big\{ \sum_{i}a_{i}\beta_{i} \Big| a_{i} \in (-1, 0]\Big\}
\subset X(T)_{\mathbb{R}} := X(T)\otimes_{\mathbb{Z}}\mathbb{R} .
$$
and let $\overline{\Sigma}$ to be the closure of $\Sigma$.
Now for $\epsilon\in\mathbb R^n$ parallel to
$\overline{\Sigma}$, 
let $\overline{\Sigma}_\epsilon:=\bigcup_{r\ge 0}\overline{\Sigma}\cap(r\epsilon+\overline{\Sigma})$;
$\epsilon\in\mathbb R^n$ is said to be {\it generic} for $\overline{\Sigma}$
if it is a non-zero vector which is parallel to $\overline{\Sigma}$ but not to any of its boundary
faces.
%
%
%
%
The following is proved in \v{S}penko and Van den Bergh:

\begin{theorem}[{\cite[Theorem 1.19]{SV}}]\label{thm:NCCRofSVdB}
Let $G=T$ be a torus and assume its action on $V$ is quasi-symmetric and generic.
For any $\epsilon\in X(T)_{\mathbb R}$ which is generic for
$\overline{\Sigma}$. Let
$$\mathcal{L}:=X(T)\cap (1/2)\overline{\Sigma}_\epsilon,\quad
U:=\bigoplus_{\chi\in\mathcal L}V(\chi),$$
where $V_\chi$ is the representation of $T$ with highest 
weight $\chi$.
Let $M=M_R^G(U)$. Then $\mathrm{End}_{R^T}(M)$ is an NCCR for $R^T$.
\end{theorem}

From now on, we assume
$T$ is the one-dimensional torus $k^{\times}$.
In this special case, we shall see that the conditions in the above Theorem \ref{thm:NCCRofSVdB}
can be described explicitly by the weights of $T$ (see Proposition
\ref{prop:equivalencecond}), and the NCCR of $R^T$ has some
very nice properties.

To this end, let us suppose $V$ is an $n$-dimensional representation of $T$.
Suppose the induced action of $T$ on the graded polynomial ring
$R :=k[V]= k[x_{1}, x_{2}, \cdots, x_{n}]$ is given by $t\circ x_i:=t^{\chi_i}x_i$
for $i=1,\cdots, n$, where $\mathrm{deg}(x_{i}) = 1$.
The weights of the action are denoted by
$$
\alpha_{T} = (\chi_{1}, \chi_{2}, \cdots, \chi_{n}) \in \mathbb{Z}^{n}.
$$
Then
by the above definition, the action of $T$ on $V$ is quasi-symmetric if
and only if the corresponding
weights $(\chi_{1}, \chi_{2}, \cdots, \chi_{n})$ satisfy
$\sum^{n}_{i = 1}\chi_{i} = 0$; that is, this action is unimodular. Recall that an
action of the group $G$ on a vector space $W$ is called
{\it unimodular} if the action of $G$ on the volume form $\bigwedge^{\mathrm{dim}(W)} W$ of
$W$ is trivial (c.f. \cite[\S1.6]{SV}). In terms of weights we have the following.

\begin{lemma}\label{equivalent}
Suppose $T=k^\times$ acts on $V$, with the induced action on $R:=k[V]$.
If
the action of $T$ on $V$ is generic and unimodular, then
the weights of the action of $T$ on
$R=k[V]$,
denoted by $\chi := (\chi_{1}, \chi_{2}, \cdots, \chi_{n}) \in \mathbb{Z}^{n}$,
has at least two positive and two negative components such that
$\sum^{n}_{i = 1}\chi_{i} = 0$.
\end{lemma}

\begin{proof}
We only need to show
$\chi= (\chi_{1}, \chi_{2}, \cdots, \chi_{n})$
has at least two positive and two negative components,
since $\sum_{i=1}^n\chi_i=0$ holds
automatically.

Suppose there were at most one positive
component or at most one negative component in $\chi$.
Without loss of generality, assume that there is one positive component in $\chi$, say
$\chi_{1}$. Then $\chi_2 , \chi_3 , \cdots, \chi_n$
are all non-positive. Since $\sum_{i=1}^n\chi_i=0$,
there are some negative components in $\chi$,
say $(\chi_{m}, \chi_{m+1}, \cdots, \chi_n)$.

Now, for any point $(0, x_2, x_3, \cdots, x_n ) \in \mathbb{C}^n \cong V$ such that $x_i \neq 0$
for all $i$
and all $t \in T$, we have
\begin{align*}
 t \circ (0, x_2 , x_3 , \cdots, x_n )&
 = (0, t^{- \chi_2} x_2, t^{- \chi_3}x_3, \cdots, t^{- \chi_n}x_n )\\
& = (0, x_2 , \cdots, x_{m-1}, t^{-\chi_m} x_{m}, \cdots,  t^{- \chi_n}x_n),
\end{align*}
where we have used that $(\chi_{2}, \chi_{3}, \cdots, \chi_{m-1})$ are all zero.
Thus 
if $t$ converges to zero, the point
$$(0, x_2, \cdots, x_{m-1}, t^{-\chi_m} x_{m}, \cdots,  t^{- \chi_n}x_n)$$
converges to $(0, x_2 , \cdots, x_{m-1}, 0, \cdots, 0)$, 
but does not reach $(0, x_3 , \cdots, x_{m-1}, 0, \cdots, 0)$.
Thus the $T$-orbit of $(0, x_2, x_3, \cdots, x_n )$ is not closed, which means the
point is in $\mathrm{Spec}(R) \backslash \mathrm{Spec}(R)^{s}$.
On the other hand, the set
$$
S' := \{(0, x_2, x_3, \cdots, x_n ) \in V = \mathrm{Spec}(R) \mid x_i \neq 0, \forall i
\}
$$
is of codimension one in $\mathrm{Spec}(R)$.
Since
$S' \subseteq \mathrm{Spec}(R) \backslash \mathrm{Spec}(R)^{s}$, we have 
$$\mathrm{codim}(\mathrm{Spec}(R) \backslash \mathrm{Spec}(R)^{s}) \leq 1.$$
This means the action of $T$ on $V$ is not generic, which is a contradiction.
\end{proof}

In \cite[\S1.6]{SV} and \cite[Theorem 8.9]{V2}, \v{S}penko and Van den Bergh show 
that if the action of $T$ on
$\mathrm{Spec}(R)$ is generic and unimodular,
then $R^{T}$ is a Gorenstein algebra. Thus
in the case $T=k^\times$,
Theorem \ref{thm:NCCRofSVdB} specializes to the following proposition 
(see \cite[Theorem 8.9]{V}):

\begin{proposition}\label{SVT}
Let $T$ be the one-dimensional torus $k^{\times}$
and $V$ be a finite dimensional representation of $T$
which is generic and unimodular.
Let
$$
\mathcal{L}:= X(T) \cap (1/2)\overline{\Sigma}_{\circ}, 
\quad U := \bigoplus_{\chi \in \mathcal{L}} V_\chi,
$$
where
$\overline\Sigma$
is the closed interval in $\mathbb R$ whose endpoints
are the negations of the sums of the positive and negative weights of $T$
respectively,
and $\overline\Sigma_\circ$ is obtained from $\overline\Sigma$
by removing the leftmost endpoint.
Then $\mathrm{End}_{R^{T}}(M)$ is an NCCR of $R^T$, where  $R = k[V]$ and
$M := M_R^T(U)$.
\end{proposition}

\begin{notation}
In what follows, we use $\Lambda$ to denote the NCCR of $R^T$ in the above proposition.
\end{notation}

In Proposition \ref{SVT}, let
$$
\varphi_{\chi}: \mathrm{End}_{\mathrm{mod} \,R^T}\big(M^T_{R}(U)\big) 
= \mathrm{End}_{\mathrm{mod} \,R^T}
\big(\bigoplus_{\chi \in \mathcal{L}} M^T_{R}(V_{\chi})\big)
\rightarrow \mathrm{End}_{\mathrm{mod} \,R^T}\big(M^T_{R}(V_{\chi})\big)
$$
be the canonical projection from $\mathrm{End}_{\mathrm{mod} \,R^T}\big(M^T_{R}(U)\big)$
to $\mathrm{End}_{\mathrm{mod} \,R^T}\big(M^T_{R}(V_{\chi})\big)$,
and
$$
\eta_{\chi}: \mathrm{End}_{\mathrm{mod} \,R^T}\big(M^T_{R}(V_{\chi})\big) \rightarrow k
$$
be the augmentation
from $\mathrm{End}_{\mathrm{mod} \,R^T}\big(M^T_{R}(V_{\chi})\big)$
to $k$. 
Then 
$ k^{\oplus \# \mathcal{L}}$ is equipped with a  $\Lambda$-module structure,
which is given by the
homomorphism
$$p: \mathrm{End}_{\mathrm{mod} \,R^T}\big(M^T_{R}(U)\big)
= \mathrm{End}_{\mathrm{mod} \,R^T}\big(\bigoplus_{\chi \in \mathcal{L}}
M^T_{R}(V_{\chi})\big) \to k^{\oplus \# \mathcal{L}}$$
such that
\begin{equation}\label{defofp}
\xymatrixcolsep{3pc}\begin{split}
\xymatrix{
& \bigoplus_{\chi \in \mathcal{L}} \mathrm{End}_{\mathrm{mod} \,R^T}\big(M^T_{R}(V_{\chi})\big)
\ar[dr]^-{ \bigoplus_{\chi \in \mathcal{L}} \eta_{\chi}}  \\
 \mathrm{End}_{\mathrm{mod} \,R^T}\big(M^T_{R}(U)\big) \ar[ur]^-{\bigoplus_{\chi \in \mathcal{L}} \varphi_{\chi}} \ar[rr]^{p} & &
 k^{\oplus \# \mathcal{L}}   }
\end{split}\end{equation}
is commutative; that is,
$$
p = \big(\bigoplus_{\chi \in \mathcal{L}} \eta_{\chi}\big) \circ \big(\bigoplus_{\chi \in \mathcal{L}} \varphi_{\chi}\big).
$$
In \S\ref{subsect:grading} below we shall show that $p$ is in fact an algebra homomorphism.

The following theorem is obtained by Iyama and Reiten in
\cite[Theorem 3.7]{IR} and is explicitly stated by
Wemyss in \cite[Theorem 4.6.3]{W2}, which says that $\Lambda$
is a Calabi-Yau algebra in the following sense:

\begin{theorem} \label{Torus}
If $S$ is an equicodimensional Gorenstein normal domain over $k$ with Krull dimension $n-1$,
and $\Lambda$ is an NCCR of $S$, then
$$
\mathrm{Hom}_{D^{b}(\Lambda)}(X, Y) \cong D\mathrm{Hom}_{D^{b}(\Lambda)}(Y, X[n-1])
$$
for all $X \in D_{f}^{b}(\Lambda)$ and $Y \in D^{b}(\Lambda)$,
where $D_{f}^{b}(\Lambda)$ is the full subcategory of $D^{b}(\Lambda)$
consisting of objects whose homologies are finite dimensional, and $D(-) := \mathrm{Hom}_{k}(-, k)$.
\end{theorem}

Plugging $X = k^{\oplus \# \mathcal{L}}$ and
$Y = \Lambda = \mathrm{End}_{\mathrm{mod} \,R^T}\big(M^T_{R}(U)\big)$
in Theorem \ref{Torus}, we obtain the following.

\begin{corollary}\label{Coro1}
Under the assumptions of Theorem \ref{Torus}, we have
$$
\mathrm{Hom}_{D^{b}(\Lambda)}(k^{\oplus \# \mathcal{L}}, \Lambda)
\cong D \mathrm{Hom}_{D^{b}(\Lambda)}(\Lambda, k^{\oplus \# \mathcal{L}}[n-1])
\cong k^{\oplus \# \mathcal{L}}[-n+1].
$$
\end{corollary}

\subsection{AS-regular algebras}
In this subsection we show that 
for $T=k^\times$,
the NCCR $\Lambda$ of $R^T$
is an Artin-Schelter regular algebra.
Let us first recall its definition.

\begin{definition}\label{AS-alg}
Suppose $A$ is a non-negatively graded algebra with $A_0$ semi-simple over $k$. $A$ is called an
{\it Artin-Schelter (AS-) regular}
algebra of dimension $d$ with Gorenstein parameter $a$
if the following two conditions hold:
\begin{enumerate}
\item[(1)] the global dimension $\mathrm{gldim}(A) = d$, and

\item[(2)] there is an isomorphism
$$\underline{\mathrm{Ext}}_{\mathrm{mod} \,A}^{i}(A_{0}, A) \cong \left\{
\begin{array}{cl}
D(A_{0})(a),& {i      =     d,}\\
0,& {i      \neq     d}
\end{array} \right.$$
in $\mathrm{grmod}\,A_0$, where $D(-) = \mathrm{Hom}_{k}(-, k)$ as before.
\end{enumerate}
\end{definition}

If an algebra $A$ satisfies (2) in above definition
and has finite injective dimension in both $\mathrm{Grmod}\,A$ and $\mathrm{Grmod}\,A^{op}$,
then we call it a graded {\it Gorenstein algebra} of dimension $d$ with Gorenstein parameter $a$.

Suppose $A$ is a Noetherian Gorenstein algebra of global dimension $d$.
If we endow $A$ with a $\mathbb{Z}$-grading such that $A_{0}$ is finite dimensional,
one may ask whether $A$ is an AS-regular algebra of dimension $d$ with
Gorenstein parameter $a$,
for some $a \in \mathbb{Z}$.
The answer is the following.

\begin{proposition}\label{Propo0}
Suppose $A$ is a unital Noetherian algebra of global dimension $d$ over 
a finite dimensional $k$-algebra $K$
with an augmentation map $\varepsilon: A \rightarrow K$ 
such that $A$ satisfies the Gorenstein condition
\begin{equation}\label{eq:Gorensteincond}
\mathrm{Ext}_{\mathrm{mod} \,A}^{i}(K, A) \cong \left\{
\begin{array}{cl}
K,     & {i      =     d},\\
0 ,     & {i      \neq     d}.
\end{array} \right.
\end{equation}
If $A$ is endowed with a $\mathbb{Z}$-grading
such that $A_{0} = K$ is a direct sum of $k$ and $A_{j} = 0$ for $j < 0$, then there is a decomposition
$$
A_{0} = \bigoplus_{j} A_{0, j}
$$
in $\mathrm{grmod}\,A_{0}$ such that $A_{0, j} \cong k$ as vector spaces
and a sequence of numbers $a_{1}, a_{2}, \cdots, a_{r}$
in $\mathbb{Z}$ such that
$$
\mathrm{\underline{Ext}}_{\mathrm{mod} \,A}^{i}(A_{0, j}, A) \cong \left\{
\begin{array}{cl}
A_{0, j}(a_{j}) ,& {i      =     d},\\
0 ,& {i      \neq     d}
\end{array} \right.
$$
in $\mathrm{grmod}\,A_{0}$.
\end{proposition}

\begin{proof}
Let $Q^{\bullet}$ be a projective resolution of $A_{0}$ in
$\mathrm{grmod}\,A$. Then $Q^{\bullet}$ is bounded
with length at most $d$ by \cite[Theorem 12]{SE}.
Now we view $Q^{\bullet}$ as a finitely generated projective resolution of
$A_{0}$ in $\mathrm{mod}\, A$; to distinguish it from the former module structure, 
we denote it 
by $\widetilde{Q^{\bullet}}$.
Then we have
\begin{align*}
\mathrm{Ext}^{i}_{\mathrm{mod} \,A}(A_{0}, A) &\cong
\mathrm{Ext}^{i}_{\mathrm{mod} \,A}(\widetilde{Q^{\bullet}}, A)\\
&\cong \mathrm{\underline{Ext}}^{i}_{\mathrm{mod} \,A}(Q^{\bullet}, A)\\
&\cong \bigoplus_{q} \mathrm{Ext}_{\mathrm{grmod}\,A}^{i}\big(Q^{\bullet}, A(q)\big)\\
&\cong \bigoplus_{q} \mathrm{Ext}_{\mathrm{grmod}\,A}^{i}\big(A_{0}, A(q)\big)
\end{align*}
in $\mathrm{grmod}\, A_{0}$ for any $i \in \mathbb{Z}$.
Since $A$ satisfies the Gorenstein condition \eqref{eq:Gorensteincond}
with $K=A_{0}$, we have
$$
A_{0} \cong \underline{\mathrm{Ext}}_{\mathrm{mod} \,A}^{d}(A_{0}, A)
\cong \bigoplus_{b} \mathrm{Ext}_{\mathrm{grmod}\,A}^{d}\big(A_{0}, A(b)\big),
$$
where $b \in \mathbb{Z}$
such that $\mathrm{Ext}_{\mathrm{grmod}\,A}^{d}\big(A_{0}, A(b)\big) \neq 0$.

We thus get a sequence numbers $b_{1}, b_{2}, \cdots, b_{r}$ in
$\mathbb{Z}$ as above such that
\begin{equation}\label{eq:decomp1}
A_{0}
\cong \bigoplus^{r}_{l = 1} \bigoplus^{m_{l}}_{s = 1} A_{0, l_{s}}
\end{equation}
and
\begin{equation}\label{eq:decomp2}
\bigoplus^{m_{l}}_{s = 1} A_{0, l_{s}}\cong
\mathrm{Ext}_{\mathrm{grmod}\,A}^{d}\big(A_{0}, A(b_{l})\big)
\cong \mathrm{Ext}_{\mathrm{grmod}\,A}^{d}
\Big(\bigoplus^{r}_{l = 1} \bigoplus^{m_{l}}_{s = 1} A_{0, l_{s}}, A(b_{l})\Big)
\end{equation}
in $\mathrm{grmod}\, A_{0}$, for any $1 \leq l \leq r$, where each $A_{0,l_s}\cong k$.
Since both isomorphisms \eqref{eq:decomp1}
and \eqref{eq:decomp2} hold in $\mathrm{grmod}\, A_{0}$, we have
$$
\mathrm{Ext}_{\mathrm{grmod}\,A}^{d}\big(A_{0, l_{s}}, A(b_{l})\big) \cong A_{0, l_{s}}
$$
and
$$
\mathrm{Ext}_{\mathrm{grmod}\,A}^{d}\big(A_{0, l_{s}}, A(b_{l'})\big) \cong 0
$$
for any $l \neq l'$. Moreover, from the decomposition
\eqref{eq:decomp1}
we get an injection
$$
\mathrm{\underline{Ext}}_{\mathrm{mod} \,A}^{i}(A_{0, l_{s}}, A) \hookrightarrow \mathrm{\underline{Ext}}_{\mathrm{mod} \,A}^{i}(A_{0}, A),
$$
from which we obtain
$$
\mathrm{\underline{Ext}}_{\mathrm{mod} \,A}^{i}(A_{0, l_{s}}, A) = 0
$$
for any $i \neq d$. Thus, we have
$$
\mathrm{\underline{Ext}}_{\mathrm{mod} \,A}^{i}(A_{0, l_{s}}, A) \cong \left\{
\begin{array}{cl}
A_{0, l_{s}}(-b_{l}) ,& {i      =     d},\\
0 ,& {i      \neq     d}
\end{array} \right.
$$
in $\mathrm{grmod}\,A_{0}$.
\end{proof}

Now, we consider the two categories $\mathrm{mod}\, R^{T}$ and $\mathrm{mod}(R, T)$,
where $\mathrm{mod}(R, T)$ is the category of $T$-equivariant $R$-modules.
In general, $\mathrm{mod}\,R^{T}$ and $\mathrm{mod}(R, T)$ are not equivalent.
However, in the case of reflexive modules, we have the following lemma.

\begin{lemma}[{\cite[Lemma 3.3]{SV}}] \label{Equi}
Under the assumptions of Proposition \ref{SVT},
let $\mathrm{ref}(R, T)$ be the category of $T$-equivariant $R$-modules
which are reflexive as $R$-modules and $\mathrm{ref}(R^{T})$ be the category
of reflexive $R^{T}$-modules. Then the following two functors
\begin{equation*}
\mathrm{ref}(R, T) \to \mathrm{ref}(R^{T}),\,
M \mapsto M^{T}
\end{equation*}
and
\begin{equation*}
\mathrm{ref}(R^{T}) \to \mathrm{ref}(R, T),\,
N \mapsto (R \otimes_{R^{T}} N)^{\ast \ast}
\end{equation*}
are mutually inverse equivalences between the two symmetric monoidal categories,
where $(-)^{\ast}:=\mathrm{Hom}_{\mathrm{mod} \,R}(-, R)$.
\end{lemma}

If $M^{T}_{R}(V_{1})$ and $M^{T}_{R}(V_{2})$
are both Cohen-Macaulay $R^{T}$-modules, then they are both reflexive $R^{T}$-modules.
Thus by the above lemma, we obtain that
\begin{align}
\mathrm{Hom}_{\mathrm{mod}
\,R^{T}}\big (M^{T}_{R}(V_{1}), M^{T}_{R}(V_{2})\big)
&\cong \mathrm{Hom}_{\mathrm{mod}(R, T)}
(R \otimes V_{1}, R \otimes V_{2}) \nonumber\\
&\cong   M^{T}_{R}
\big(\mathrm{Hom}_{\mathrm{mod} \,k}(V_{1}, V_{2})\big)\label{eq:RGmoduleiso}
\end{align}
in $\mathrm{mod}\, R^{T}$.

Moreover, observe that
$$
M^{T}_{R}(V) = (V \otimes R)^{T} \cong\big (D\big(D(V)\big) \otimes R\big)^{T}
\cong \mathrm{Hom}_{\mathrm{Rep}(T)}\big(D(V), R\big)
$$
for any finite-dimensional $V \in \mathrm{Rep}(T)$,
where $\mathrm{Rep}(T)$ is the category of representations of $T$.
If $V$ is an irreducible representation of $T$, then
the elements of $M^{T}_{R}(V)$ are in one to one correspondence with
the irreducible sub-representations of $T$ in $R$ that are isomorphic to
$D(V)$ in $\mathrm{Rep}(T)$.

\subsubsection{Gradings of $R^T$ and $\Lambda$}\label{subsect:grading}

In Proposition \ref{SVT}, both $R^{T}$ and $\Lambda$ have a natural
grading given as follows:
suppose that the weights of the action $T$ on $R$ are $(\chi_1, \chi_2, \cdots, \chi_n)$.
Since $R^{T}$ is generated by the elements in 
$\{x^{m_1}_1 x^{m_2}_2 \cdots x^{m_n}
\in R | \sum^{n}_{i = 1} m_i \chi_i =0 \}$
as a vector space, $R^T$ endows a grading induced from $R$; moreover, it is straightforward to see that
the multiplication on $R^T$ respects this grading.

Recall that
$$
M = M^T_{R}(U) = (U \otimes R)^{T} = \bigoplus_{\chi \in \mathcal{L}} (V_\chi \otimes R)^{T}.
$$
Since $(V_\chi \otimes R)^{T}$ is generated by the elements in
\begin{equation}\label{formula:grading}
\Big\{x^{m_1}_1 x^{m_2}_2 \cdots x^{m_n}_n
\in R\Big | \sum^{n}_{i = 1} m_i \chi_i = -\chi \Big\}
\end{equation}
which have non-negative grading, we endow $M$ with a grading induced from $R$,
which has no negatively graded component and is
compatible with the grading of $R^T$. Thus, $M$ is a {\it graded} $R^T$-module.
Since $\Lambda = \mathrm{End}_{\mathrm{mod} \,R^T}(M)$ and $M$ is a finitely generated graded $R^T$-module,
we have
$$
\mathrm{End}_{\mathrm{mod} \,R^T}(M) \cong \bigoplus_{i \in \mathbb{Z}}
\mathrm{Hom}_{\mathrm{grmod}\,R^T}\big(M, M(i)\big)
$$
as vector spaces. Let
$\big(\mathrm{End}_{\mathrm{mod} \,R^T}(M)\big)_{i}
:=\mathrm{Hom}_{\mathrm{grmod}\,R^T}\big(M, M(i)\big)$;
then $\mathrm{End}_{\mathrm{mod} \,R^T}(M)$ is a graded vector space,
which is also preserved under the product (composition).
Thus $\Lambda$ is a graded algebra.

Now, we describe the non-positive components of $\Lambda$.
First, for any $\alpha, \beta \in \mathcal{L}$,
by the construction of the NCCR of $R^T$,
both $M^{T}_{R}(V_\alpha)$ and $M^{T}_{R}(V_\beta)$ are reflexive
$R^T$-modules.
Thus we have the following isomorphism
\begin{eqnarray*}
M^{T}_{R}(V_{\beta - \alpha})
&\cong&
M^T_R(V_{-\alpha}\otimes V_{\beta})
 \cong M^{T}_{R}\big(D(V_\alpha) \otimes V_\beta\big)\\
& \cong& \mathrm{Hom}_{\mathrm{mod}(k, T)}(V_{\alpha}, V_{\beta} \otimes R)\\
& \cong& \mathrm{Hom}_{\mathrm{mod}(R, T)}(V_{\alpha} \otimes R, V_{\beta} \otimes R) \\
& \stackrel{(-)^{T}}\cong&
\mathrm{Hom}_{\mathrm{mod} \,R^T}\big(M^{T}_{R}(V_{\alpha}), M^{T}_{R}(V_{\beta})\big)
\quad\textup{(by Lemma \ref{Equi}).}
\end{eqnarray*}
Since each of the above equalities is grading preserving,
we obtain the following isomorphism
$$
M^{T}_{R}(V_{\beta - \alpha})
\stackrel{\cong}{\rightarrow}
\mathrm{Hom}_{\mathrm{mod} \,R^T}\big(M^{T}_{R}(V_{\alpha}), M^{T}_{R}(V_{\beta})\big)
$$
as graded $R^{T}$-modules.

Observing that
the left hand side of the above isomorphism has
no negatively graded component (see \eqref{formula:grading}),
we obtain that $(\mathrm{Hom}_{\mathrm{mod} \,R^T}\big(M^{T}_{R}(V_{\alpha})\big)_i = 0$
for $i < 0$, which implies that $\Lambda_{i} = 0$ for $i < 0$.

From \eqref{eq:RGmoduleiso} and the above argument,
we have a graded $R^T$-module isomorphism
$$
 \bigoplus_{\alpha, \beta \in \mathcal{L}} M^{T}_{R}(V_{\beta - \alpha})
 \stackrel{\cong}{\rightarrow} \bigoplus_{\alpha, \beta \in \mathcal{L}}
 \mathrm{Hom}_{\mathrm{mod} \,R^T}\big(M^{T}_{R}(V_{\alpha}), M^{T}_{R}(V_{\beta})\big)
 \cong \Lambda,
$$
which implies that $\Lambda_{0} =\big (\bigoplus_{\alpha,
\beta \in \mathcal{L}} M^{T}_{R}(V_{\beta - \alpha})\big)_0 \cong
\big((R^{T})^{\oplus \sharp \mathcal{L}}\big)_0 = k^{\oplus \sharp \mathcal{L}}$.

From the above we also see that the algebra morphism
$p : \Lambda \rightarrow k^{\oplus \sharp \mathcal{L}}$
in \eqref{defofp}
gives the
augmentation map of $\Lambda$.

\subsubsection{The AS-regular algebra structure}

We are now ready to show the NCCR $\Lambda$
is an AS-regular algebra. Let us start with the following.

\begin{proposition}\label{ToGo}
Under the assumptions of Proposition \ref{SVT}, we have the following:
\begin{enumerate}
\item[$(1)$]
$R^{T}$ is a Noetherian graded Gorenstein algebra, and moreover,

\item[$(2)$] it is
of dimension $n-1$ with Gorenstein parameter $n$.
\end{enumerate}
\end{proposition}

\begin{proof}

(1) It is straightforward to see that $R^T$ is Noetherian,
and is also graded
with the grading induced from $R$.
Moreover, under the two
conditions of Proposition \ref{SVT} (the action of $T$
being generic and unimodular),
$R^T$ is a Gorenstein algebra.

Now let $R^{\bullet}$ be the graded projective resolution of $k$ in the
derived category $D\big(\mathrm{grmod}\,R^{T}\big)$.
Write $R^{\bullet}$ as $\widetilde{R^{\bullet}}$ when we view it as
the projective resolution of $k$ in $D(R^{T})$. We have
\begin{align*}
k &\cong\mathrm{Ext}_{\mathrm{mod} \,R^T}^{n-1}(k, R^T)
\cong
\mathrm{Ext}_{\mathrm{mod} \,R^T}^{n-1}(\widetilde{R^{\bullet}}, R^T)
\cong
\underline{\mathrm{Ext}}_{\mathrm{mod} \,R^T}^{n-1}(R^{\bullet}, R^T) \\
&
\cong
\underline{\mathrm{Ext}}_{\mathrm{mod} \,R^T}^{n-1}(k, R^T)
\cong
\bigoplus_{i}
\mathrm{Ext}_{\mathrm{grmod}\,R^T}^{n-1}\big(k, R^T(i)\big),
\end{align*}
where the first equality comes from the Gorensteinness of $R^T$.
Since $k$ is a one-dimensional vector space,
there is a number $r \in \mathbb{Z}$ such that
$$
\mathrm{Ext}_{\mathrm{grmod}\,R^T}^{n-1}\big(k, R^T(-r)\big)= k, 
$$
and
$$
\mathrm{Ext}_{\mathrm{grmod}\,R^T}^{n-1}\big(k, R^T(i)\big) = 0,
$$
for any $i \neq -r$.

In the meantime, since $R^{T}$ is a Noetherian Gorenstein algebra,
the injective dimension of $R^T$ in $\mathrm{Mod}\,R^T$ is finite, which
implies the finiteness of the injective dimension of $R^T$ in $\mathrm{Grmod}\,R^T$.
Thus by definition, $R^T$ is a
Noetherian graded Gorenstein algebra with Gorenstein parameter $r$.

(2) First, it is obvious that
$$\mathrm{dim}(R^{T}) = \mathrm{dim}\big(\mathrm{Spec}(R^{T})\big)
= \mathrm{dim}\big(\mathrm{Spec}(R)\big) - \mathrm{dim}(T) =  n-1.$$
Now, Orlov \cite[Theorems 16, 24 \& Proposition 28]{DR2} shows that
the Gorenstein parameter of $R^{T}$ is equal to the number $r \in \mathbb{Z}$ such that
$\mathcal{O}_{X}(-r) \cong \omega_{X}$,
where $X$ is the quotient stack $\big[\big(\mathrm{Spec}(R^{T})
\setminus \{ 0\}\big) /k^{\times}\big]$
and the action of $k^{\times}$ on $\mathrm{Spec}(R^{T})$ is given by the grading on $R^{T}$
(see also the argument after the proof of \cite[Theorem 25]{DR2}).
Thus to prove the proposition, it suffices to prove that $\mathcal{O}_{X}(-n) \cong \omega_{X}$,
which is equivalent to showing that the weight of the $k^{\times}$-action on
the sections of the canonical line bundle on $X$ is $n$. Observe that the latter is
also equal to the degree of the section of the canonical line bundle
on $Y := \mathrm{Spec}(R^T)$, where the degree is induced from the grading of $R^T$.

On the other hand, observe that the canonical
line bundle $\omega_Z$ on $Z := \mathrm{Spec}(R) = k^n$ has a section
\begin{equation}\label{eq:formulafors}
s:=dx_{1} dx_{2} \cdots dx_{n} : Z \longrightarrow \omega_Z.
\end{equation}
Since the representation $k^n$ of $T$
is unimodular, the induced action
of $T$ on the canonical line bundle on $\mathrm{Spec}(R)$ is trivial.
This implies that $s$ induces a section $\tilde{s}$ of the canonical
line bundle $\omega_Y \cong \rho_{\ast} \omega_{Z}$ on $\mathrm{Spec}(R^T)$
such that the following diagram$$
\xymatrixcolsep{4pc}
\xymatrix{
Z \ar[r]^{s} \ar[d]_{\rho} & \omega_Z \ar[d]^{\rho'}  \\
Y \ar[r]^{\tilde{s}}
& \omega_{Y}
}
$$
commutes,
where $\rho: Z \rightarrow Y$ is the quotient morphism
and $\rho'$ is the lift of $\rho$ to the canonical line bundles.

Observe that
the degree of $\tilde{s}$ is equal to the
degree of $s$ and that
$s$ has degree $n$ by \eqref{eq:formulafors}; we obtain that
$\tilde s$ also has degree $n$.
\end{proof}

\begin{lemma}\label{RTequi}
Under the assumptions of Proposition \ref{SVT},
for any
$\alpha_{r} \in \mathcal{L} $, let $e_{r} \in \Lambda$ be
the indecomposable
idempotent corresponding to $M^{T}_{R}(V_{\alpha_{r}})$, which is  a direct summand of
$ \bigoplus_{\chi \in \mathcal{L}} M^{T}_{R}(V_{\chi}) \cong M^{T}_{R}(U)$. Then
$$
e_{r} \Lambda e_{r} \cong R^{T}
$$
as graded algebras.
\end{lemma}

\begin{proof}
First, we  have that $e_r \Lambda \otimes_{\Lambda} \Lambda e_r \cong e_{r} \Lambda e_{r}$
as algebras, where the multiplication of
$e_r \Lambda \otimes_{\Lambda} \Lambda e_r$ is
given as
follows: for any $a_1 \otimes_{\Lambda} b_1 , a_2 \otimes_{\Lambda}
b_2 \in e_r \Lambda \otimes_{\Lambda} \Lambda e_r$,
$$
(a_1 \otimes_{\Lambda} b_1)(a_2 \otimes_{\Lambda} b_2) : = a_1 m_{\Lambda}(b_1, a_2) \otimes_{\Lambda} b_2,
$$
where $m_\Lambda$ is the multiplication of $\Lambda$.

Since $e_r \in \Lambda$ presents the composition of
the canonical projection (denoted by $\pi_r$) from $M$ to $M^{T}_{R}(V_{\alpha_{r}})$
and the canonical inclusion (denoted by $l_r$) from $M^{T}_{R}(V_{\alpha_{r}})$ to $M$:
$$
M \stackrel{\pi_r}\twoheadrightarrow M^{T}_{R}(V_{\alpha_{r}}) \stackrel{l_r}\hookrightarrow M,
$$
we have that the graded $R^T$-module $e_{r} \Lambda$ is isomorphic to
$\mathrm{Hom}_{\mathrm{mod} \,R^T}\big(M^{T}_{R}(V_{\alpha_{r}}), M\big)$. Analogously,
the graded $R^T$-module $\Lambda e_r$ is isomorphic
to $\mathrm{Hom}_{\mathrm{mod} \,R^T}\big(M, M^{T}_{R}(V_{\alpha_{r}})\big)$
as graded $R^T$-modules. Thus, we obtain a graded algebra isomorphism
\begin{equation}\label{eq:iso111}
e_r \Lambda e_r \cong \mathrm{Hom}_{\mathrm{mod} \,R^T}\big(M^{T}_{R}(V_{\alpha_{r}})
, M\big) \otimes_{\Lambda} \mathrm{Hom}_{\mathrm{mod} \,R^T}\big(M, M^{T}_{R}(V_{\alpha_{r}})\big),
\end{equation}
where the multiplication on the right hand side is given by the composition of morphisms.

In the meantime, there is a natural algebra morphism
\begin{equation}\label{eq:iso222}
\begin{split}
\begin{array}{cr}
\vartheta_r: &\mathrm{Hom}_{\mathrm{mod} \,R^T}\big(M^{T}_{R}(V_{\alpha_{r}}), M\big)
\otimes_{\Lambda} \mathrm{Hom}_{\mathrm{mod} \,R^T}
\big(M, M^{T}_{R}(V_{\alpha_{r}})\big)\\
&\rightarrow \mathrm{Hom}_{\mathrm{mod} \,R^T}\big(M^{T}_{R}(V_{\alpha_{r}}), 
M^{T}_{R}(V_{\alpha_{r}})\big),
\end{array}
\end{split}
\end{equation}
which is given by composition of morphisms.
We show $\vartheta_r$ is an isomorphism.
On the one hand, for any
$f \in \mathrm{Hom}_{\mathrm{mod} \,R^T}\big(M^{T}_{R}(V_{\alpha_{r}}), M^{T}_{R}(V_{\alpha_{r}})\big)$,
we have $\vartheta_{r}\big((l_r \circ f) \otimes_{\Lambda} \pi_r\big) = f$.
Thus $\vartheta_r$ is surjective.
On the other hand, suppose we have
$\vartheta_{r}\big(\sum_{j}(a_j \otimes_{\Lambda} b_j)\big)
= \sum_{j}b_j \circ a_j = 0$. Since
\begin{align*}
\sum_{j}a_j \otimes_{\Lambda} b_j
&= \sum_{j}a_j \otimes_{\Lambda} (\pi_r \circ l_r \circ b_j)\\
&= \sum_{j}\big((l_r \circ  b_j) \circ a_j\big) \otimes_{\Lambda} \pi_r\\
&= \sum_{j}\big(l_r \circ  (b_j \circ a_r)\big) \otimes_{\Lambda} \pi_r,
\end{align*}
we have that $\sum_{j}a_j \otimes_{\Lambda} b_j= 0$ and hence $\vartheta_r$ is injective.
We thus get $\vartheta_r$ is an isomorphism,
and combining \eqref{eq:iso111} and \eqref{eq:iso222}, we obtain 
\begin{equation}\label{eq:iso333}
e_r \Lambda e_r \cong
\mathrm{Hom}_{\mathrm{mod} \,R^T}\big(M^{T}_{R}(V_{\alpha_{r}}), M^{T}_{R}(V_{\alpha_{r}})\big).
\end{equation}

Now, since $V_{\alpha_{r}}$ is a
one dimensional representation of $T$ with weight $\alpha_{r}$
for any $\alpha_{r} \in \mathcal{L}$,
we have $D(V_{\alpha_{r}}) = \mathrm{Hom}_{k}(V_{\alpha_{r}}, k)= V_{- \alpha_{r}}$
and $V_{- \alpha_{r}} \otimes V_{\alpha_{r}}\cong V_0\cong k$, where
$k$ is viewed as the trivial representation of $T$,
and therefore
$$
M^T_R\big(D(V_{\alpha^r})\otimes V_{\alpha^r}\big)
\cong M^T_R(V_0)\cong (k\otimes R)^T=R^T.$$
Thus by \eqref{eq:RGmoduleiso} we have the following graded algebra isomorphism
\begin{equation}\label{eq:iso444}
R^T = M^{T}_{R}\big(D(V_{\alpha_{r}}) \otimes V_{\alpha_{r}} \big)
\stackrel{\cong}{\rightarrow} \mathrm{Hom}_{\mathrm{mod} \,R^T}
\big(M^{T}_{R}(V_{\alpha_{r}}), M^{T}_{R}(V_{\alpha_{r}})\big).
\end{equation}
Combining \eqref{eq:iso333} and \eqref{eq:iso444},
we get the graded algebra isomorphism $e_r \Lambda e_r \cong R^T$.
\end{proof}


We next show the main result of this subsection.

\begin{theorem}\label{ToAS}
Under the assumptions of Proposition \ref{SVT},
let $\Lambda$ be the NCCR of $R^T$. Then
$\Lambda$ is an AS-regular algebra of dimension $n-1$ with Gorenstein parameter $n$.
\end{theorem}

\begin{proof}
Recall that from \S\ref{subsect:grading},
$\Lambda$ has a grading satisfying Proposition \ref{Propo0}
and the morphism $p : \Lambda \rightarrow k^{\oplus \sharp \mathcal{L}}$ is the
augmentation map.
By Theorem \ref{Torus} and Propositions \ref{Propo0} and \ref{ToGo},
we only need to prove
$$
a_{1} = a_2 = \cdots= a_{r} = n,
$$
where $\{a_{1}, a_{2}, \cdots, a_{r}\}$ is given in Proposition \ref{Propo0}.

To this end, let $\Lambda$ be the graded algebra $A$ in Proposition \ref{Propo0},
and $Q_{i}^{\bullet}$ be the projective resolution of $\Lambda_{0, i} \cong e_{i} \Lambda_{0}$
of a graded $\Lambda$-module, where $e_{i}$ is the idempotent corresponding to $\Lambda_{0, i}$.
From the proof of Proposition \ref{Propo0}, we see
$$
\mathrm{Ext}_{\mathrm{mod} \,\Lambda}^{n-1}( \Lambda_{0, i},  \Lambda) \cong  \Lambda_{0, i}(a_{i}).
$$
Therefore we have isomorphisms
\begin{align*}
k(a_{i})& \cong \mathrm{Ext}_{\mathrm{mod} \,\Lambda}^{n-1}( \Lambda_{0, i},  \Lambda)\\
&  \cong  \mathrm{Hom}^{n-1}_{K^{b}(\mathrm{grmod}\, \Lambda)}(Q_{i}^{\bullet}, \Lambda)\\
& \cong  \mathrm{Hom}^{n-1}_{K^{b}(\mathrm{grmod}\, \Lambda )}(Q_{i}^{\bullet}, e_{i}\Lambda)\\
& \cong \mathrm{Hom}^{n-1}_{K^{b}(\mathrm{grmod}\, \Lambda )}\big(Q_{i}^{\bullet},
\mathrm{Hom}_{e_{i} \Lambda e_{i}} (\Lambda e_{i}, e_{i}\Lambda e_{i})\big)\\
& \cong \mathrm{Hom}^{n-1}_{K^{b}(\mathrm{grmod}\, e_{i} \Lambda e_{i})}(Q_{i}^{\bullet}
\otimes_{\Lambda} \Lambda e_{i}, e_{i} \Lambda e_{i})\\
& \cong \mathrm{Hom}^{n-1}_{K^{b}(\mathrm{grmod}\, e_{i} \Lambda e_{i})}(Q_{i}^{\bullet}
e_{i}, e_{i} \Lambda e_{i}).
\end{align*}
Moreover, since $(Q_{i}^{\bullet})^{j} \in \mathrm{add}(\Lambda)$ for any $j$,
$(Q_{i}^{\bullet})^{j}e_{i}$ is a Cohen-Macaulay $R^{T}$-module.
Observe that when $R^{T}$ is a Gorenstein algebra, $M$ is a Cohen-Macaulay $R^{T}$-module
if and only if
\begin{equation}\label{eq:HomMS}
\mathrm{Hom}^{i}_{D^{b}(R^T)}(M, R^T) = 0,
\end{equation}
for any $i \neq 0$.
From the fact that $R^{T}$ is a Noetherian graded Gorenstein algebra of dimension
$n-1$ with Gorenstein parameter $n$, we get an isomorphism:
\begin{eqnarray*}
\mathrm{Hom}^{n-1}_{K^{b}(\mathrm{grmod}\, e_{i} \Lambda e_{i})}(Q_{i}^{\bullet} e_{i}, e_{i} \Lambda e_{i})
 &\stackrel{\eqref{eq:HomMS}}\cong&
\mathrm{Hom}^{n-1}_{D^{b}(\mathrm{grmod}\, e_{i} \Lambda e_{i})}(Q_{i}^{\bullet} e_{i}, e_{i} \Lambda e_{i})\\
& \cong &\mathrm{Hom}^{n-1}_{D^{b}(\mathrm{grmod}\,R^{T})}(k, e_{i} \Lambda e_{i}) \\
& \cong& k(n)
\end{eqnarray*}
in $\mathrm{grmod}\,k$,
where the second isomorphism holds since we have
\begin{eqnarray*}
 Q_{i}^{\bullet} e_{i} \cong
 Q_{i}^{\bullet} \otimes^{L}_{\Lambda} \Lambda e_{i}
 \cong e_{i} \Lambda_{0} \otimes^{L}_{\Lambda} \Lambda e_{i}
 \cong e_{i} \Lambda_{0}
 \cong k
\end{eqnarray*}
 in $D^{b}(\mathrm{grmod}\,e_{i} \Lambda e_{i})$.
From these two isomorphisms we get that $a_{i} = n$ for $i=1,\cdots, r$.
\end{proof}

\subsection{Graded isolated singularity}

In this subsection we give a characterization of $R^T$ as a graded
isolated singularity.
Let us start with the following definition, which is
introduced by Mori and Ueyama in \cite{MU2}:

\begin{definition}[{\cite[Definition 1.1]{MU2}}]\label{Gra-iso}
A right Noetherian graded algebra $A$ is called a {\it graded isolated singularity} if
$\mathrm{gldim}(\mathrm{tails}\, A) < \infty$, where
$$
\mathrm{gldim}(\mathrm{tails}\, A) := \mathrm{sup}
\left\{i\in\mathbb Z \Big| \mathrm{Hom}^{i}_{D^b(\mathrm{tails}\, A)}(M, N) \neq 0, M, N \in \mathrm{tails}\, A\right\}.
$$
\end{definition}

In op. cit. the authors also obtained the following result
(see\cite[Theorem 3.10]{MU2}):
suppose a finite group $\Gamma$ acts on an AS-regular algebra $S$ of dimension $d \geq 2$,
which has homological determinant 1,
and let $\bm e := \frac{1}{|\Gamma|} \sum_{\gamma \in\Gamma} (1,\gamma)$;
then $(S \ast \Gamma) /\langle\bm e\rangle$ is finite dimensional over $k$
if and only if $S^{\Gamma}$ is a graded isolated singularity and
$$
\psi: S \ast \Gamma \rightarrow
\mathrm{End}_{\mathrm{mod} \,S^{\Gamma}}(S),\, (s, \gamma)
\mapsto \{t \mapsto s\gamma(t)\}
$$
is an isomorphism of graded algebras.

\begin{notation}\label{def:e}
From now on, we let $e$ be the idempotent of $\Lambda$
which corresponds to the summand $M^{G}_{R}(V_l)$, where $l$
is the minimal number in $\mathcal{L}$ given in Proposition \ref{SVT}.
\end{notation}

The following theorem generalizes the above result of Mori and Ueyama
to the case when $\Gamma$ is the one-dimensional torus $k^\times$.

\begin{theorem}\label{Sing}
Let $T$ be $k^\times$ which acts 
on $R = k[V]= k[x_{1}, x_{2}, \cdots, x_{n}]$ with weights 
$\chi = (\chi_{1}, \chi_{2}, \cdots, \chi_{n})$.
Suppose the action of $T$ is generic and unimodular.
Then the following are equivalent:
\begin{enumerate}
\item[$(1)$] $\mathrm{gcd}(\chi_i,\chi_j)=1$ whenever
$\chi_i\chi_j<0$ for all $1\le i,j\le n$;

\item[$(2)$] the functor $( - )e: \mathrm{tails}\,
\Lambda \rightarrow \mathrm{tails}\, R^{T}$ is an equivalence functor
with inverse $( - ) \otimes_{e \Lambda e} e \Lambda$;

\item[$(3)$] $R^{T}$ is a graded isolated singularity;

\item[$(4)$] $\mathrm{Spec}(R^{T})$ has a unique isolated singularity at the origin.
\end{enumerate}
\end{theorem}

Before proving the theorem, let us first recall the following
lemma of Mori and Ueyama:

\begin{lemma}[{\cite[Lemma 3.17]{MU}, \cite[Lemma 2.7]{MU2}}]\label{Idemp}
Let $A$ be a Noetherian graded algebra and $e$ be an idempotent of $A$
such that $eAe$ is also a Noetherian graded algebra and that $Ae$ and $eA
\in \mathrm{grmod}\,eAe $
are finitely generated left $eAe$-modules.
Then $A/\langle e\rangle \in \mathrm{tors}\, A$
if and only if
$$
- \otimes_{A} Ae: \mathrm{tails}\, A \rightarrow \mathrm{tails}\,eAe
$$
is an equivalence functor.
\end{lemma}

\begin{proof}[Proof of Theorem \ref{Sing}]
From $(2)$ to $(1)$:
Without loss of
generality we assume
$\mathrm{gcd}(\chi_{1}, \chi_{2}) = h > 1$
and $\chi_{1} > 0 > \chi_{2}$.
From the construction of $\Lambda = \mathrm{End}_{R^{T}}
(\bigoplus_{\lambda  \in \mathcal{L}} M^{T}_{R}(V_{\lambda}))$
in Proposition \ref{SVT}, we may choose an indecomposable element $e' \in \Lambda_{0}$
associated to the summand $M^{T}_{R}(V_{l+1})$, where $l$ is the minimal number in $\mathcal{L}$.

Observe that we have
$l +1 \in \mathcal{L}$. In fact, if $l+1$ were not in $\mathcal{L}$, then since
$l$ is the minimal number in $\mathcal{L}$, $\mathcal{L}$ has only one point $l$.
Now, we have $\Lambda =\mathrm{End}_{R^T}
\big(M^{T}_{R}(V_l)\big) = R^T$ by Lemma \ref{Equi}. Since $\Lambda$ is
homologically smooth,
$\mathrm{Spec}(R^{T})$ is a smooth scheme. Thus the action of $T$ on $V$ is trivial,
which is a contradiction.

Since $x^{- \chi_{2}}_{1} x^{\chi_{1}}_{2} \in R^{T} \subset \Lambda$,
we have that $e' (x^{- \chi_{2}}_{1} x^{\chi_{1}}_{2})^{N} \in e' R^{T} \subset \Lambda$, for
any $N > 0$.
Now, fix $N>0$. Since
$e' (x^{- \chi_{2}}_{1} x^{\chi_{1}}_{2})^{N} = e' (x^{- \chi_{2}}_{1} x^{\chi_{1}}_{2})^{N} e'$,
we have that
$
e' (x^{- \chi_{2}}_{1} x^{\chi_{1}}_{2})^{N} \in \Lambda e \Lambda
$
if and only if
there are two elements $f_{1} \in e' \Lambda e \cong (R \otimes V_{-1})^{T}
$ and $f_{2} \in e \Lambda e' \cong (R \otimes V_{1})^{T}$ such that
$$
f_{1}f_{2} = e' (x^{- \chi_{2}}_{1} x^{\chi_{1}}_{2})^{N}.
$$
On the other hand
$$
f_{1}f_{2} = e' (x^{- \chi_{2}}_{1} x^{\chi_{1}}_{2})^{N}
$$
if and only if there are two numbers $\alpha, \beta \in \mathbb{Z}$ such that
$- \chi_{2}N \geq \alpha \geq 0$, $\chi_{1}N \geq \beta \geq 0$ and
$f_{2} = e x^{\alpha}_{1} x^{\beta}_{2} e' $, which is equivalent to that
$$
\chi_{1} \alpha + \chi_{2} \beta = -1
$$
since $k \cdot f_{2} \cong V_{-1}$.
However, we have assumed that $\mathrm{gcd}(\chi_{1}, \chi_{2}) = h > 1$, which is
a contradiction.
Thus we can choose infinity many $N$ making
$$
e' (x^{- \chi_{2}}_{1} x^{\chi_{1}}_{2})^{N} \notin \Lambda e \Lambda.
$$
It contradicts to the fact that $\Lambda/\Lambda e \Lambda$
is a finite dimensional vector space by Lemma \ref{Idemp}.

From $(3)$ to $(2)$:
since $R^{T}$ is a Noetherian graded Gorenstein algebra,
the result follows from \cite[Theorem 2.5]{MU2}, which
says that,
if $ A $ is a Noetherian graded Gorenstein isolated singularity of
dimension $d \geq 2$ and $N \in \mathrm{Ob}(\mathrm{CM}^{\mathbb{Z}}(A))$ such that
$N$ contains $A$ as a direct summand, then
$$
\mathrm{\underline{Hom}}_{\mathrm{tails} \,A}(N, -): \mathrm{tails}( \,A) 
\rightarrow\mathrm{tails} \,
\big(\mathrm{\underline{End}}_{\mathrm{mod} \,A}(N)\big)
$$
is an equivalence functor.

Now since
$e \Lambda e \in \mathrm{add}(\Lambda e )$,
let $A = R^T =  e \Lambda e$ and $N = \Lambda e(=M)$
(see Lemma \ref{RTequi}), then the above equivalence
holds in this case.
Since
$$
\mathrm{\underline{Hom}}_{\mathrm{tails} \,e \Lambda e}(\Lambda e, -)
\otimes_{\Lambda} \Lambda e \cong
\mathrm{\underline{Hom}}_{\mathrm{tails} \,e \Lambda e}( e \Lambda e, -)
$$
is the identity functor on $\mathrm{tails}( e \Lambda e)$,
we have that
the functor $\mathrm{\underline{Hom}}_{e \Lambda e}(\Lambda e, -)$  is inverse to
the functor $( - )e := ( - ) \otimes_{\Lambda} \Lambda e$. Hence,
$$
( - )e : \mathrm{tails}(\Lambda) \rightarrow \mathrm{tails}( e \Lambda e)
$$
is an equivalence functor.
Moreover, $\big(( - ) \otimes_{e \Lambda e} e \Lambda\big)e$ is also an identity functor
on $\mathrm{tails}( e \Lambda e)$. Since the inverse functor is unique, we then have
$$
( - ) \otimes_{e \Lambda e} e \Lambda \cong \mathrm{\underline{Hom}}_{e \Lambda e}(\Lambda e, -)
$$
on $\mathrm{tails}( e \Lambda e)$.

From $(4)$ to $(3)$:
suppose $\mathrm{Spec}(R^{T})$ has a unique isolated singularity at the origin.
Then $\mathrm{Spec}(R^{T}) \setminus \{ 0 \}$ is smooth,
and therefore the abelian
category $\mathrm{Coh}\big(\mathrm{Spec}(R^{T}) \setminus \{ 0 \} \big)$ of
coherent sheaves on $\mathrm{Spec}(R^{T}) \setminus \{ 0 \}$ has finite global dimension.
This implies that the abelian category
$\mathrm{Coh}\big(\big[\big(\mathrm{Spec}(R^{T}) \setminus \{ 0 \}\big)/k^{\times}\big]\big)$
of coherent sheaves of the quotient stack
$\big[\big(\mathrm{Spec}(R^{T}) \setminus \{ 0 \}\big)/k^{\times}\big]$
also has finite global dimension, by \cite[Lemma 3.1 (4)]{SV}.
Thus by the fact that
$\mathrm{Coh}\big(\big[\big(
\mathrm{Spec}(R^{T}) \setminus \{ 0 \}\big) /k^{\times}\big]\big)
\cong \mathrm{tails}\, R^{T}$
(see \cite[Proposition 28]{DR2}),
we have $\mathrm{gldim}(\mathrm{tails}\, R^{T}) < \infty$.

From $(1)$ to $(4)$:
recall that Luna's slice theorem in \cite{Lu} (see also the proof of \cite[Lemma 3.8]{SV}) says that
for any closed point
$p \neq \{0\} \in \mathrm{Spec}(R^{T})$, suppose it is the image of $\tilde p \in{\mathrm{Spec}(R)}^{s}$,
then there are two \'etale morphisms
$$
S//T_{\tilde p} \rightarrow N_{\tilde p}//T_{\tilde p},
$$
and
$$
S//T_{\tilde p} \rightarrow {\mathrm{Spec}(R)}//T \cong {\mathrm{Spec}(R^{T})},
$$
where $T_{\tilde p}$ is the stabilizer of $\tilde p$, $N_{\tilde p}$ is a
$T_{\tilde p}$-invariant complement to the inclusion of
$T_{\tilde p}$-representations ${T}_{\tilde p}(T)/{T}_{\tilde p}(T_{\tilde p})
\subseteq \mathrm{T}_{\tilde p}\big(\mathrm{Spec}(R)\big)$,
and $S$ is the affine $T_{\tilde p}$-invariant slice to the $T$-orbit of $\tilde p$.

Thus by Luna's slice theorem, we have
that $p$ is non-singular in $\mathrm{Spec}(R^{T})$ if and only if the origin is
non-singular in $N_{\tilde p}//T_{\tilde p}$, which is equivalent to that $T_{\tilde p}$ is trivial.
Since $\tilde p = (x_{1}, x_{2}, \cdots, x_{n}) \in \mathrm{Spec}(R)^{s}$, there are
$x_{i}, x_{j} \in\{x_{1}, x_{2}, \cdots, x_{n}\}$ such that $\chi_{i}\chi_{j} < 0$.
Therefore $\mathrm{gcd}(\chi_{i}, \chi_{j}) = 1$ by assumption.
Thus for any $t \in T_{\tilde p}$, we have $t^{\chi_{i}} = 1$ and $t^{\chi_{j}} = 1$,
from which we obtain that $t = 1 \in T$. Therefore $T_{\tilde p}$ is trivial.
\end{proof}

As a corollary to Theorem \ref{Sing},
we have the following proposition, 
which describes the singularity of a generic and unimodular
$k^\times$-action in terms
of its weights.

\begin{proposition}\label{prop:equivalencecond}
Suppose $T=k^\times$ acts on $V$, with the induced action on $R:=k[V]$.
The the following two conditions are equivalent:
\begin{enumerate}
\item[$(1)$] the action of $T$ on $V$ is generic and unimodular, and $R^T$ is
an isolated singularity;

\item[$(2)$] the weights $\chi = (\chi_{1}, \chi_{2}, \cdots, \chi_{n}) \in \mathbb{Z}^{n}$
of the action of $T$ on
$R$ has at least two positive and two negative components such that
$\sum^{n}_{i = 1}\chi_{i} = 0$ and
$\mathrm{gcd}(\chi_i,\chi_j)=1$ whenever $\chi_i\chi_j<0$ for all $1\le i,j\le n$.
\end{enumerate}
\end{proposition}

\begin{proof}
From (1) to (2):
Follows from a combination of Lemma \ref{equivalent}
and Theorem \ref{Sing}.

From (2) to (1):
Thanks to Theorem \ref{Sing}, we only need to show that
the action of $T$ on $V$ is generic and unimodular.

First, choose a closed point $x \in \mathrm{Spec}(R) \cong V$ such that
at most one of its coordinates
is zero.
We assume that
$$
x = (0, x_2 , x_3 , \cdots, x_n)
$$
with $x_i \neq 0$ for any $i$ (if all coordinates are nonzero, the proof
is similar). 
Without loss of generality we assume that $\chi_2$ is positive and $\chi_3$ is negative.
We thus have
\begin{align*}
 t \circ (0, x_2 , x_3 , \cdots, x_n ) =
  (0, t^{- \chi_2} x_2, t^{- \chi_3}x_3, \cdots, t^{- \chi_n}x_n ).
\end{align*}
(Here we should note that the weights of the action of
$T$ on $\mathbb{C}^{n} = V$ and the weights of the action of on $k[V]$ are opposite.)
If $t$ converges to zero, then
$t^{- \chi_2} x_2$ diverges to infinity and $t^{- \chi_3} x_3$ converges to zero.
On the other hand, if $t$ diverges to infinity, then
$t^{- \chi_2} x_2$ converges to zero and $t^{- \chi_3} x_3$ diverges to infinity.
Thus the $T$-orbit of $x$ is closed.

Moreover,  there are $a, b \in \mathbb{Z}$ such
that $a \chi_2 + b \chi_3 = 1$. If there is an element $t \in T$ such that $t \circ x = x$,
then
$t^{-\chi_2} = t^{-\chi_3} = 1$, and thus
$$
t = t^{a \chi_2 + b \chi_3} = (t^{-\chi_2})^{-a} (t^{\chi_3})^{-b} = 1.
$$
Therefore, the stabilizer of $x$ is trivial and thus $x \in \mathrm{Spec}(R)^{s}$.
More generally,
for any $x'$ in the set
$$
S := \{(x_1, x_2, \cdots, x_n ) \in V  \mid \mbox{at most one of $x_i$ is zero}
\},
$$
we have that $x' \in \mathrm{Spec}(R)^{s}$ by analogous argument.
Thus we have $S \subseteq \mathrm{Spec}(R)^{s}$ and
$\mathrm{Spec}(R) \backslash \mathrm{Spec}(R)^{s}
\subseteq \mathrm{Spec}(R) \backslash S$.
Moreover, it is direct to see that $\mathrm{codim}(\mathrm{Spec}(R) \backslash S) \geq 2$
by the definition of $S$.
We thus obtain
$\mathrm{codim}(\mathrm{Spec}(R) \backslash \mathrm{Spec}(R)^{s}) \geq 2$,
and hence the action of $T$ on $R$ is generic.

Second,
observe that the action of $T$
on the volume form is
\begin{align*}
t \circ (dx_1 dx_2 \cdots dx_n)& = (t^{\chi_1}dx_1) (t^{\chi_2}dx_2) \cdots (t^{\chi_n}dx_n)\\
&=  t^{\sum^{n}_{i=1} \chi_i} dx_1 dx_2 \cdots dx_n\\
&=dx_1 dx_2 \cdots dx_n,
\end{align*}
for any $t\in T$, from which follows that
the action of $T$ on $V$ is unimodular.
\end{proof}

\section{Proof of Theorem \ref{MainTh}: the first case}\label{sect:proofofmain1}

In this section, we prove the first case of Theorem \ref{MainTh},
which is when $G$ is the one dimensional torus $k^\times$.
Let us first recall the following.

\begin{definition}[Tilting object]
Let $\mathcal{C}$ be a triangulated category. An object $X \in Ob(\mathcal{C})$
is called a {\it tilting object} if
\begin{enumerate}
\item[$(1)$] $\mathcal{C} = \mathrm{thick}(X)$, and

\item[$(2)$] $\mathrm{Hom}_{\mathcal{C}}(X, X[i]) = 0$ for any $i \neq 0$.
\end{enumerate}
\end{definition}

An important result of tilting theory is the following (see \cite{HHK} for more details).
If $X$ is a tilting object of an algebraic Krull-Schmidt triangulated $\mathcal C$, then the following functor
$$
\mathcal C\to K^{b}\left(\mathrm{proj}\big(\mathrm{End}_{\mathcal C}(X)\big)\right),
\, M\mapsto \mathrm{Hom}_{\mathcal C}(X, M)
$$
is an equivalence of triangulated categories.

In the following two subsections, we first
construct an object, as described in \S\ref{sect:Intro},
which generates
$D_{sg}^{gr}(R^T)$ (see Proposition \ref{fully});
then we show that the $\mathrm{Ext}$-groups of the generator
vanish except in degree zero
(see Proposition \ref{ExtVani}). 

\subsection{Existence of the generator}

Let us go back to diagram \eqref{diag:first}
(see also \cite{MU} for more details). 
Here we assume $G=T=k^\times$.
Recall that in the diagram,  $\nu: D^b(\mathrm{grmod}\,R^G)\to
D_{sg}^{gr}(R^G)$ is the Verdier localization functor, and
$\Omega^i_{\Lambda}(-)$ is the $i$-th syzygy functor.
Also, as has already been done before, in what follows we shall write the functors
$$(-)\otimes_\Lambda\Lambda e: 
D^b(\mathrm{grmod}\,\Lambda)\to D^b(\mathrm{grmod}\,R^G)\quad\mbox{and}\quad
D^b(\mathrm{tails}\,\Lambda)
\to D^b(\mathrm{tails}\, R^G)$$
simply as $(-)e$.

\begin{definition}\label{def:EQ}
Let $Q$ be the following set of objects in $D_{sg}^{gr}(R^{T})$:
$$
Q := \left\{
\begin{array}{l}\nu\big((1-e) \Lambda e\big),
\nu\Big(\big(\Omega^{1}_{\Lambda}\big((1-e )\Lambda_0\big)(1)\big) e\Big),
\nu \Big(\big(\Omega^{2}_{\Lambda}\big((1-e)\Lambda_0\big)(2)\big)
e\Big),\\
\cdots, \nu\Big(\big(\Omega^{n-1}_{\Lambda}\big((1-e)\Lambda_0\big)(n-1)\big) e\Big)
\end{array}
\right\},
$$
and let $E_Q$ be the direct sum of the objects in $Q$.
\end{definition}

In \cite[Theorem 1.4]{MU},
Mori and Ueyama
construct a tilting object for $D_{sg}^{gr}(S^G)$,
where $S$ is a Noetherian Koszul AS-regular algebra
and $G$ is a group of finite elements with homological determinant 1.
The above definition is inspired by and analogous to their construction.
Note that in our case, by Theorem \ref{ToAS} the NCCR $\Lambda$ for $R^T$
is an AS-regular algebra of {\it dimension $n-1$} with {\it Gorenstein
parameter $n$}, and the first object in $Q$, which
does not appear in loc. cit., is to compensate
the discrepancy between these two parameters.
The following is the main result of this subsection.

\begin{proposition}\label{fully}
Under the assumptions of Theorem \ref{MainTh},
let $E_Q$ be given in Definition \ref{def:EQ}.
Then we have
$$
D_{sg}^{gr}(R^{T}) = \mathrm{thick}(E_{Q}).
$$
In other words, $E_Q$, as an object in $D_{sg}^{gr}(R^{T})$,
generates $D_{sg}^{gr}(R^{T})$ itself.
\end{proposition}

To prove the proposition, we first introduce several notions.
Remind that in diagram \eqref{diag:first}, $\Phi: D_{sg}^{gr}(R^T)\to
D^b(\mathrm{tails}\,R^T)$
and $\pi: D^b(\mathrm{grmod}\,R^T)\to
D^b(\mathrm{tails}\,R^T)$ denote the Orlov embedding and 
the natural projection respectively.
We also need the following.

\begin{definition}[{c.f. \cite[\S 2]{KS}}]\label{def:minimalapprox}
 Let $A$ be a graded algebra, $\mathcal{C}$ be a class of
 graded $A$-modules which are closed under direct summands, and
 $M$ be a graded $A$-module.
 A {\it minimal left $\mathcal{C}$-approximation} of $M$
 is a morphism of graded $A$-modules
 $ f : M \rightarrow X\in \mathcal{C}$ such that
\begin{enumerate}
\item[(1)] for any $N \in \mathcal{C}$,
$$
\mathrm{Hom}_{\mathrm{grmod}\,A}(N, f):
\mathrm{Hom}_{\mathrm{grmod}\,A}(X, N) \rightarrow
\mathrm{Hom}_{\mathrm{grmod}\,A}(M, N),\,
g\mapsto g\circ f
$$
is surjective, and
\item[(2)]  if $g \in \mathrm{End}_{\mathrm{grmod}\,A}(X)$ satisfies $f = g \circ f$ , then $g$ is an automorphism.
\end{enumerate}
The minimal right $\mathcal{C}$-approximation
is defined dually.
\end{definition}


Now
let $A$ in the above definition 
be the graded commutative ring $R^{T} \cong e \Lambda e$, and
$\mathcal{C}$ be $\mathrm{add}(e \Lambda e)$, which is the category consisting
of direct summands of finite direct sums of $e \Lambda e$.
Since $e \Lambda e$ is a graded locally finite-dimensional algebra with
$(e \Lambda e)_{0} = k$, we have that 
for any graded $e \Lambda e$-module $M$ and any $N\in\mathrm{add}(e\Lambda e)$,
$$\mathrm{Hom}_{e\Lambda e}(M, N)\cong \mathrm{Hom}_{e\Lambda e}(M, e\Lambda e)
\otimes_{e\Lambda e}N,$$
and hence
the minimal left $\mathcal{C}$-approximation
exists for $M$ and is unique up to isomorphism 
by \cite{KS} (see also \cite[Definition 5.1]{DW}).

\begin{notation}
From now on, we let
$L_{e \Lambda e} M : =\mathrm{Cone}(f)$
be the cone of the minimal left $\mathcal{C}$-approximation
of $M$ in the triangulated category $D^{b}(\mathrm{grmod}\,e \Lambda e)$.
\end{notation}


\begin{lemma}\label{lemma:Tocom}
With the notations from diagram \eqref{diag:first}, when $G=k^\times$ we have
\begin{align*}
\Phi(E_{Q})& \cong \Phi \circ \nu\left( L_{e \Lambda e} \left(\Big(\bigoplus^{n-1}_{i = 1}
\big(\Omega_{\Lambda}^{i}
\big((1-e)\Lambda_0\big)(i)\big)
e \Big)\oplus (1 - e)\Lambda e\right)\right)\\
& \cong
\pi\left( L_{e \Lambda e} \left(\Big(\bigoplus^{n-1}_{i = 1}
\big(\Omega_{\Lambda}^{i} \big((1-e)\Lambda_0\big)
(i)\big) e\Big) \oplus (1 - e)\Lambda e\right)\right).
\end{align*}
\end{lemma}

In order to prove this lemma, let us
denote by $(\mathrm{grmod}\,\Lambda)_{\geq 0}$
the full subcategory of $\mathrm{grmod} \, A$
whose objects consist of $N \in \mathrm{grmod} \, \Lambda$
such that $N_{i} = 0$ for $i < 0$. Moreover,
denote by $D^{b}\big((\mathrm{grmod}\,\Lambda)_{\geq 0}\big)$
the full subcategory of $D^{b}(\mathrm{grmod}  \, \Lambda)$ whose objects consist of
$M^{\bullet} \in D^{b}(\mathrm{grmod}  \, \Lambda)$
such that $M^{j} \in \mathrm{Ob}\big
((\mathrm{grmod}\,\Lambda)_{\geq 0}\big)$ for $j \in \mathbb{Z}$.
The following lemma is due to Amiot \cite{CA}.

\begin{lemma}[{\cite[Theorem 4.3]{CA}}]\label{Cancommu}
Let $A$ be a Noetherian graded Gorenstein algebra and $M^{\bullet} \in D^{b}(\mathrm{grmod} \, A)$.
If 
\begin{enumerate}\item[$(1)$]
$M^{\bullet} \in D^{b}\big((\mathrm{grmod}\,A)_{\geq 0}\big)$ and

\item[$(2)$]
$
\mathrm{Hom}^{r}_{D^{b}(\mathrm{grmod}\, A)}\big(M^{\bullet}, A(i)\big) = 0
$
for any $i \leq 0$ and $r \in \mathbb{Z}$,
\end{enumerate}
then
$$
\pi(M^{\bullet}) = \Phi \circ \nu(M^{\bullet}).
$$
\end{lemma}

\begin{proof}[Proof of Lemma \ref{lemma:Tocom}]
The first equality holds since
$$
E_{Q} = \nu\left( L_{e \Lambda e}
\left(\Big(\bigoplus^{n-1}_{i = 1}\big(\Omega_{\Lambda}^{i}
\big((1-e)\Lambda_0\big)(i)\big) 
e \Big)\oplus (1 - e)\Lambda e\right)\right)
$$
in $D_{sg}^{gr}(R^{T})$.

To prove the second equality, we only need to show
$$L_{e \Lambda e}\left(\Big(\bigoplus^{n-1}_{i = 1}
\big(\Omega_{\Lambda}^{i} \big((1-e)\Lambda_0\big)(i)\big) e\Big)
\oplus (1 - e)\Lambda e\right)$$
satisfies the two conditions of
Lemma \ref{Cancommu}; that is, we need to show the following:

\begin{claim}\label{claim22}
$L_{e \Lambda e}\left(\Big(\bigoplus^{n-1}_{i = 1}\big(\Omega_{\Lambda}^{i} \big(
(1-e)\Lambda_0\big)(i)\big) e\Big)
\oplus (1 - e)\Lambda e\right)\in D^b\big((\mathrm{grmod}\,\Lambda)_{\ge 0}\big)$.
\end{claim}

\begin{claim}\label{claim33}
$\mathrm{Hom}^{r}_{D^b(\mathrm{grmod}\, A)}
\left(L_{e \Lambda e}\left(\Big(\bigoplus^{n-1}_{i = 1}\big(\Omega_{\Lambda}^{i}\big(
 (1-e)\Lambda_0\big)(i)\big) e\Big)
\oplus (1 - e)\Lambda e\right), \Lambda(j)\right)$ is zero 
for any $j\le 0$ and $r \in \mathbb{Z}$.
\end{claim}

Since the proofs of these two claims are lengthy,
we defer their proofs to the end of this subsection.
Admitting these two claims,
Lemma \ref{lemma:Tocom} is proved.
\end{proof}

We continue the preparation of the proof of Proposition
\ref{fully}.
Since $\Lambda$ is an AS-regular algebra and
a Cohen-Macaulay $R^{T}$-module,
$(1-e)\Lambda_0$ has a bounded graded projective $\Lambda$-module resolution, denoted by $P^{\bullet}$, with length
$n$ such that 
\begin{equation}\label{def:Presol}
\begin{split}
&P^{1-n} = (1-e)\Lambda(-n),\quad
P^{0} = (1-e)\Lambda,\\
&P^{i}\in \mathrm{add}_{j \in [-n+1, -1]}\big(\bigoplus_{i} e_{i} \Lambda(j)\big)
\;\mbox{for}\, -n +1 \leq i \leq -1.
\end{split}
\end{equation}
(In fact, $P^i$ is given by the minimal right $\mathrm{add}(\Lambda)$-approximation
of
$\ker d_{i+1}: P^{i+1}\to P^{i+2} $
in the sense of Definition \ref{def:minimalapprox} above.)

\begin{lemma}\label{claim1}
With the notations from \eqref{diag:first}, we have that
$\pi \big((1-e) \Lambda e(j)\big)$
is an element in 
\begin{equation}\label{eq:induction0} 
\mathrm{thick}
\left(\pi\left(\Big( \bigoplus^{n-1}_{i = 1}
\big(\Omega_{\Lambda}^{i}\big( (1-e)\Lambda_0\big)(i)\big) e\Big)
\oplus (1 - e)\Lambda e\right), \pi\Big( \bigoplus^{-n+1}_{r = 0} e \Lambda e(r)\Big)\right),
\end{equation}
for any $-n+1 \leq j \leq 0$.
\end{lemma}

\begin{proof}
We show the proof by induction. First, observe that
$
\pi \big((1-e) \Lambda e (j)\big)
$
is an object in
$$\mathrm{thick}\left(\pi\left(\Big( \bigoplus^{n-1}_{i = 1}
\big(\Omega_{\Lambda}^{i} \big((1-e)\Lambda_0\big)(i) e
\big)\Big)
\oplus (1 - e)\Lambda e\right), \pi\Big( \bigoplus^{-n+1}_{r = 0} e \Lambda e(r)\Big)\right)
$$
for $j = 0, -1$,
since $ \pi \big(\Omega_{\Lambda}^{n-1} 
\big((1-e)\Lambda_0\big)(n-1) e \big)\cong
\pi \big((1-e) \Lambda e (-1)\big)$.

Next, we assume the claim holds for $l\leq j\leq 0$; that is,
$
\pi \big((1-e) \Lambda e (j)\big) 
$
is an object in
$$
\mathrm{thick}\left(\pi\left(\Big( \bigoplus^{n-1}_{i = 1}
\big(\Omega_{\Lambda}^{i}\big( (1-e)\Lambda_0\big)(i) \big)\Big)
\oplus (1 - e)\Lambda e\right), \pi\Big( \bigoplus^{-n+1}_{r = 0} e \Lambda e(r)\Big)\right)
$$
for $l \leq j \leq 0$ and some $-n+1 \leq l \leq -1$.
Consider the following exact sequence of graded $\Lambda$-modules
$$
\xymatrixcolsep{1.2pc}
\xymatrixrowsep{.8pc}
\xymatrix{
0 \ar[r] &P^{1-n}(n+l-1) \ar@{=}[d]\ar[r]&\ar[r]
P^{2-n}(n+l-1)&
\ar[r] \cdots&\ar[r] \Omega_{\Lambda}^{l+n-1} (1-e)(n+l-1)  &0,\\
&(1-e)\Lambda (l-1) &&&
}
$$
which implies that
\begin{align*}
0 \longrightarrow \pi \big((1-e)\Lambda e\big) (l-1)& \longrightarrow 
\pi\big (P^{2-n}e(n+l-1)\big)
\longrightarrow \cdots \\
&\longrightarrow \pi\big(\Omega_{\Lambda}^{l+n-1}\big( (1-e)\Lambda_0\big)(n+l-1)
e\big)
\longrightarrow 0
\end{align*}
is isomorphic to $0$
in $D^{b}(\mathrm{tails}\, R^{T})$.
Now observe that $\pi \big(P^{2-n}e(n+l-1)\big)$
is the direct sum of some summands of
$\pi \big(\Lambda e (l)\big)$ and $\pi\big(\Lambda e(l+1)\big)$.
Indeed, in the resolution $P^{\bullet}$, for any $-i$ and any summand
$e' \Lambda(r')$ of $P^{-i}$
and summand $e'' \Lambda(r'')$ of $P^{-i+1}$ , the composition
$$
e' \Lambda(r') \hookrightarrow P^{-i}
\xrightarrow{d_{-i}} P^{-i+1} \twoheadrightarrow e'' \Lambda(r'')
$$
is zero unless $r' < r''$, since $P^{-i}$ is constructed by
the minimal right $\mathrm{add}(\Lambda)$-approximation of $\mathrm{Ker}(d_{-i+1})$.
Moreover, since $\Lambda$ is an AS-regular algebra of dimension $n-1$ with Gorenstein parameter $n$,
we have that the length of $P^{\bullet}$ is $n-1$ and $P^{1-n}$ is
the direct sum of some summands of $\Lambda (-n)$.
Hence, we obtain that for any $-i$, $P^{-i}$ is the direct sum
of some summands of $\Lambda (-i-1)$ and $\Lambda(-i)$.

In a similar way, we have that
$\pi
\big(P^{-l-n+1}e(n+l-1)\big)$
is the direct sum of some summands of $ \pi\big(\Lambda e (-1)\big)$ and $\pi(\Lambda e)$.
Combining these two statements, we have
$$
\pi \big(P^{r}e(n+l-1)\big) \in
\mathrm{thick}\left(\pi \Big(\bigoplus^{0}_{i = l} (1-e) \Lambda e(i)\Big),
\pi\Big(\bigoplus^{0}_{r = -n+1} e \Lambda e(r)\Big)\right)
$$
for any $2-n \leq r \leq l-n-1$. Hence from
$
\mathrm{thick}\Big(\pi\big(\bigoplus^{0}_{i = l} (1-e) \Lambda e(i)\big)\Big)
$
contained in
$$
\mathrm{thick}\left(\pi\left( \Big(\bigoplus^{n-1}_{i = 1}\big
(\Omega_{\Lambda}^{i} \big((1-e)\Lambda_0\big)(i)\big) e
\Big) \oplus (1 - e)\Lambda e\right),
\pi\Big(\bigoplus^{0}_{r = -n+1} e \Lambda e(r)\Big)\right),
$$
we have that
$
\pi \big(P^{r} e(n+l-1)\big)
$
is an object in
$$\mathrm{thick}\left(\pi\left( \Big(\bigoplus^{n-1}_{i = 1}
\big(\Omega_{\Lambda}^{i}\big( (1-e)\Lambda_0\big)(i)\big)
 e\Big)
\oplus (1 - e)\Lambda e\right), \pi\Big( \bigoplus^{0}_{r = -n+1} e \Lambda e(r)\Big)\right)
$$
for any $2-n \leq r \leq l-n-1$. Therefore we obtain that
$
\pi \big((1-e)\Lambda e\big) (l-1)
$
is an object in
$$ 
\mathrm{thick}\left(\pi\left(\Big( \bigoplus^{n-1}_{i = 1}
\big(\Omega_{\Lambda}^{i}\big( (1-e)\Lambda_0\big)(i)\big) e\Big)
\oplus (1 - e)\Lambda e\right), \pi\Big( \bigoplus^{0}_{r = -n+1} e \Lambda e(r)\Big)\right),
$$
and the lemma follows.
\end{proof}

We also recall the following lemma, due to Orlov:

\begin{lemma}[{\cite[Theorem 16]{DR2}}]\label{Orlovembed}
Suppose $A$ is a Gorenstein algebra of dimension $d$ with positive Gorenstein parameter $a$.
If $A$ is Noetherian, then
there is a fully faithful functor $\Phi: D^{gr}_{sg}(A) \rightarrow D^{b}(\mathrm{tails}\, A)$ and
a semi-orthogonal decomposition
$$
D^{b}(\mathrm{tails}\,A) = \langle \pi A(-a+1), \pi A(-a+2), \cdots, \pi A, \Phi D^{gr}_{sg}(A) \rangle.
$$
\end{lemma}

\begin{proof}[Proof of Proposition \ref{fully}]
By Lemma \ref{Orlovembed} and Proposition \ref{ToAS}, we have
$$
D^{b}(\mathrm{tails}\, \Lambda) 
= \mathrm{thick}\left(\pi\Big( \bigoplus^{0}_{i = -n+1} \Lambda(i)\Big)\right).
$$
Applying $(-)e$ in Theorem \ref{Sing} to both sides of the above identity, we get
$$
D^{b}(\mathrm{tails}\, R^{T}) = \mathrm{thick}\left( \pi\Big(\bigoplus^{0}_{i = -n+1} 
\Lambda e(i)\Big)\right).
$$
Now by Lemma \ref{Orlovembed}, applying $\mu$
to both sides of the above identity we further obtain
$$
D_{sg}^{gr}(R^{T}) = \mathrm{thick}\left( \bigoplus^{0}_{i = -n+1} \mu \circ \pi \big((1-e) \Lambda e(i)\big)\right).
$$
Thus to prove the proposition it suffices to prove that
\begin{equation}\label{eq:mupi}
\mu \circ \pi \big((1-e)\Lambda e(j)\big) \in \mathrm{thick}(E_{Q}),
\end{equation}
for any $-n+1 \leq j \leq 0$.

In fact, since $\mu \circ \Phi$ is the identify functor on $D_{sg}^{gr}(R^{T})$, we have
\begin{align*}
&\nu \left(L_{e \Lambda e} \left(\Big(\bigoplus^{n-1}_{i = 1}\big(\Omega_{\Lambda}^{i}
\big((1-e)\Lambda_0\big)(i)\big) e\Big)
\oplus (1 - e)\Lambda e\right)\right)\\
& \cong \mu \circ \Phi \circ \nu \left( L_{e \Lambda e} 
\left(\Big(\bigoplus^{n-1}_{i = 1}
\big(\Omega_{\Lambda}^{i} \big((1-e)\Lambda_0\big)(i)\big) e 
\Big)\oplus (1 - e)\Lambda e\right)\right) \\
&\cong  \mu \circ \pi\left(L_{e \Lambda e} \left(\Big(\bigoplus^{n-1}_{i = 1}
\big(\Omega_{\Lambda}^{i}\big((1-e)\Lambda_0\big)(i)\big)  e
\Big) \oplus (1 - e)\Lambda e\right)\right),
\end{align*}
where the last equality follows from Lemma \ref{lemma:Tocom}.
Hence, we have 
\begin{align}\label{eq:EQ}
E_Q&=
\nu \left(L_{e \Lambda e} \left(\Big(\bigoplus^{n-1}_{i = 1}
\big(\Omega_{\Lambda}^{i}\big((1-e)\Lambda_0\big)
(i)\big) e \Big)\oplus (1 - e)\Lambda e\right) \right)\nonumber\\
&\cong \mu \circ \pi \left( L_{e \Lambda e} 
\left(\Big( \bigoplus^{n-1}_{i = 1}
\big(\Omega_{\Lambda}^{i}\big((1-e)\Lambda_0\big)
(i)\big) e\Big) \oplus (1 - e)\Lambda e\right)\right).
\end{align}
Now applying $\mu$ to both sides of \eqref{eq:induction0},
we get $$\mu\circ\pi \big((1-e) \Lambda e(j)\big)
\in \mathrm{thick}
\left(\mu\circ\pi\left(\Big( \bigoplus^{n-1}_{i = 1}\big(\Omega_{\Lambda}^{i} \big((1-e)\Lambda_0\big)
(i)\big) e\big)\Big)
\oplus (1 - e)\Lambda e\right)\right),$$
where the right hand side is equivalent to
$\mathrm{thick}(E_Q)$ by \eqref{eq:EQ}.
This proves \eqref{eq:mupi}, and the proposition follows.
\end{proof}

\begin{proof}[Proof of Claim \ref{claim22}]
First, since $\Lambda$ is an AS-regular algebra of dimension $n-1$
with Gorenstein parameter $n$,
in the $(-i)$-th and $(-i+1)$-st positions of the resolution $P^{\bullet}(i)$, which are
$$
P^{-i}(i) \twoheadrightarrow \Omega^{i}_{\Lambda}(1-e)(i) \hookrightarrow P^{-i+1}(i),
$$
we have that $P^{-i}(i)$ is the direct sum of some summands of
$\Lambda (-1)$ and $ \Lambda(0)$,
and $P^{-i+1}(i)$ is the direct sum of some summands of $\Lambda(0)$ and $\Lambda (1)$.
Since $\Omega^{i}_{\Lambda}\big((1-e)\Lambda_0\big)(i)$ is a submodule of $P^{-i+1} (i)$, if
$\Omega^{i}_{\Lambda}\big((1-e)\Lambda_0\big)
(i)$ were not in $(\mathrm{grmod}\,\Lambda)_{\geq 0}$, then
the image of $\Omega^{i}_{\Lambda}\big((1-e)\Lambda_0\big)(i)_{-1}$ in $P^{-i+1} (i)$
contains a summand of $\Lambda_{0}(1)$.
But the image of
$$
d_{i}(i): P^{-i}(i) \rightarrow P^{-i+1}(i)
$$
does not contain any summand of $\Lambda_{0}(1)$
by the construction of $P^{\bullet}$. Thus
$$
\Omega^{i}_{\Lambda}\big((1-e)\Lambda_0\big)(i) \in (\mathrm{grmod}\,\Lambda)_{\geq 0}
$$
for any $ n-1 \geq i \geq 1$.

Moreover, we see that $(1-e) \Lambda \in (\mathrm{grmod}\,\Lambda)_{\geq 0}$.
Thus tensoring with $(-) \otimes_{\Lambda} \Lambda e$, we have
$$
\Big(\bigoplus^{n-1}_{i = 1}\big(\Omega_{\Lambda}^{i} \big((1-e)\Lambda_0\big)
(i)\big) e\Big)
\oplus (1 - e)\Lambda e \in (\mathrm{grmod}\,R^{T})_{\geq 0}.
$$
Applying the minimal left $\mathrm{add}(e \Lambda e)$-approximation of the above $R^T$-module,
 we also have
\[
L_{e \Lambda e}\left(\Big(\bigoplus^{n-1}_{i = 1}
\big(\Omega_{\Lambda}^{i}\big((1-e)\Lambda_0\big)(i)\big) e\Big)
\oplus (1 - e)\Lambda e\right) \in (\mathrm{grmod}\,R^{T})_{\geq 0}.
\qedhere\]
\end{proof}

In order to prove Claim \ref{claim33}, we first recall the following
lemma, due to Mori:

\begin{lemma}[{\cite[Lemma 2.9]{M}}]\label{Depth2}
Let $A$ be a right Noetherian graded algebra with $A_0$ semi-simple over $k$. Then the natural morphism
$$
\mathrm{\underline{Hom}}_{\mathrm{grmod}\, A}(M, N)
\rightarrow \mathrm{\underline{Hom}}_{\mathrm{tails}\,A}(M, N)
$$
is an isomorphism of vector spaces
for any $N, M \in \mathrm{grmod}\, A$ with $\mathrm{depth} (N) \geq 2$.
\end{lemma}

\begin{proof}[Proof of Claim \ref{claim33}]
Since
$$
\mathrm{Hom}^{l}_{D^{b}(\mathrm{grmod}\, R^{T})}
\big((1-e) \Lambda e, e \Lambda e(j)\big) = 0
$$
for any $j \leq 0$ and $l \in \mathbb{Z}$, and
$$
\mathrm{Hom}^{l}_{D^{b}(\mathrm{grmod}\, R^{T})}
\left(L_{e\Lambda e}\big(\big(\Omega_{\Lambda}^{i}\big((1-e)\Lambda_0\big)(i)
\big) e\big), e \Lambda e(j)\right) = 0
$$
for any $l \leq -2$, $1\le i\le n-1$ and $j\le 0$,
we have to show that
$$
\mathrm{Hom}^{l}_{D^{b}(\mathrm{grmod}\, R^{T})}
\left(L_{e\Lambda e}\big(\big(\Omega_{\Lambda}^{i} \big((1-e)\Lambda_0\big)
(i)\big) e\big), e \Lambda e(j)\right) = 0
$$
for any $l\ge -1$, $ 1\le i\le n-1$ and $j \leq 0$.

In the meantime, from the proof of Lemma \ref{claim1}, we have that
$P^{-i}(i)$ is the direct sum of some summands of $\Lambda(-1)$ and $\Lambda$.
Thus, we have
$$
\mathrm{Hom}^{l}_{D^{b}(\mathrm{grmod}\, R^{T})}
\left(L_{e\Lambda e}\big(\big(\Omega_{\Lambda}^{i} \big((1-e)\Lambda_0\big)(i) \big)e\big), 
e \Lambda e(j)\right) = 0
$$
for any $l\ge -1$, $1\le i\le  n-1$ and $j \leq -2$.

Thus, combining the above cases, to show the claim, we only need to show that
$$
\mathrm{Hom}^{l}_{D^{b}(\mathrm{grmod}\, R^{T})}
\left(L_{e\Lambda e}\big(\big(\Omega_{\Lambda}^{i}\big((1-e)\Lambda_0\big)(i)\big) e\big), e \Lambda e(j)\right) = 0
$$
for any $l\ge -1$, $1\le i\le n-1$ and $-1 \leq j \leq 0$.
We divide the proof into several cases.

\noindent{\bf Case 1: $l\ge 1$, $1\le i\le n-1$ and $j=-1$ or $0$.}
We give $\Omega_{\Lambda}^{i} \big((1-e)\Lambda_0\big)(i)$ the following resolution
$$
\xymatrixcolsep{1.2pc}
\xymatrixrowsep{.8pc}
\xymatrix{
0 \ar[r]&P^{-n+1}(i) \ar[r]\ar@{=}[d]& \cdots \ar[r]&P^{-i-1}(i)
\ar[r]& P^{-i}(i) \ar[r]& \Omega_{\Lambda}^{i} (1-e)(i) \ar[r]& 0\\
& (1 - e)\Lambda(i-n) &&&&&
}
$$
from $P^{\bullet}$, which is
denoted by $P^{\bullet}_{-i}(i)$.
Since $\Lambda e$ is a Cohen-Macaulay $R^{T}$-module, we have
$$
\mathrm{Hom}^{r}_{D^{b}(\mathrm{grmod}\;
R^{T})}\left(\big(\Omega_{\Lambda}^{i} \big((1-e)\Lambda_0\big)(i) \big)
e, e \Lambda e(j)\right) =
\mathrm{Hom}^{r}_{K^{b}(\mathrm{grmod}\;
R^{T})}\left(P^{\bullet}_{-i}(i) e, e \Lambda e(j)\right)
$$
for any $r \in \mathbb{Z}$, where $K^b(-)$ means the corresponding bounded homotopy category.
Moreover, since $\Lambda$ is an AS-regular algebra, we have
\begin{align*}
&\mathrm{Hom}^{l}_{K^{b}(\mathrm{grmod}\, R^{T})}\big(P^{\bullet}_{-i}(i) e, e \Lambda e(j)\big) \\
& \cong \mathrm{Hom}^{l+i}_{K^{b}(\mathrm{grmod}\,R^{T})}\big(P^{\bullet}(i)e, e \Lambda e(j)\big) \\
&\cong  \mathrm{Hom}^{l+i}_{D^{b}(\mathrm{grmod}\, R^{T})}\big(P^{\bullet}(i)e, e \Lambda e(j)\big)\\
&\cong \mathrm{Hom}^{l+i}_{D^{b}(\mathrm{grmod}\, R^{T})}\big((1-e)
\otimes^{L}_{\Lambda} \Lambda e(i), e \Lambda e(j)\big)\\
&\cong \mathrm{Hom}^{l+i}_{D^{b}(\mathrm{grmod}\, R^{T})}\big((1-e)e(i), e \Lambda e(j)\big)\\
&\cong 0
\end{align*}
for $l \geq 1$. Thus we have
$$
\mathrm{Hom}^{l}_{D^{b}(\mathrm{grmod}\, R^{T})}
\left(\big(\Omega_{\Lambda}^{i} \big((1-e)\Lambda_0\big)(i) \big)e, e \Lambda e(j)\right) = 0
$$
for $l \geq 1$.
Applying the minimal left $\mathrm{add}(e \Lambda e)$-approximation of 
$\big(\Omega_{\Lambda}^{i}\big( (1-e)\Lambda_0\big)(i)\big) e$, we have
$$
\mathrm{Hom}^{l}_{D^{b}(\mathrm{grmod}\, R^{T})}
\left(L_{e\Lambda e}\big(\big(\Omega_{\Lambda}^{i}\big((1-e)\Lambda_0\big)(i) \big)e\big), 
e \Lambda e(j)\right) = 0
$$
for $l \geq 1$.

\noindent{\bf Case 2: $l=0$ or $-1$, $1\le i \leq n-1$ and $j=0$.}
From the fist condition in Definition \ref{def:minimalapprox}, we have
$$
\mathrm{Hom}^{l}_{D^{b}(\mathrm{grmod}\, R^{T})}
\left(L_{e\Lambda e}\big(\big(\Omega_{\Lambda}^{i}\big((1-e)\Lambda_0\big)(i) 
\big)e\big), e \Lambda e\right) = 0,
$$
for any $i=1,\cdots, n-1$ and $l =0$ or $-1$.

\noindent{\bf Case 3: $l=-1$, $1\le i\le n-1$ and $j=-1$.}
Since
$P^{-i}e(i)$ is a direct sum of some summands of $\Lambda e (-1)$ and $ \Lambda e(0)$,
$$
\mathrm{Hom}^{-1}_{D^{b}(\mathrm{grmod}\, R^{T})}
\left(L_{e\Lambda e}\big(\big(\Omega_{\Lambda}^{i} \big((1-e)\Lambda_0\big)(i) \big)e\big), 
e \Lambda e(-1)\right) = 0.
$$

\noindent{\bf Case 4: $l=0$, $i=1$ and $j=-1$.}
Similarly to the above case,
$$
\mathrm{Hom}^{0}_{D^{b}(\mathrm{grmod}\, R^{T})}
\left(\big(\Omega_{\Lambda}^{1}\big((1-e)\Lambda_0)
(1)\big) e,
e \Lambda e(-1)\right) = 0.
$$

\noindent{\bf Case 5: $l=0$, $2\le i \leq n-1$ and $j=-1$.}
We have to show that,
for $i\geq 2$,
\begin{align*}
&\mathrm{Hom}^{0}_{D^{b}(\mathrm{grmod}\, R^{T})}
\left(L_{e\Lambda e}\big(\big(\Omega_{\Lambda}^{i}\big((1-e)\Lambda_0\big)
(i) \big)e\big), e \Lambda e(-1)\right)\\
&\cong
\mathrm{Hom}^{0}_{D^{b}(\mathrm{grmod}\, R^{T})}
\left(\big(\Omega_{\Lambda}^{i}\big((1-e)\Lambda_0\big)(i) \big)e, 
e \Lambda e(-1)\right)\\
& = 0.
\end{align*}
We show the equality by contradiction.
If there were a morphism
$$0 \neq s \in \mathrm{Hom}^{0}_{D^{b}(\mathrm{grmod}\, R^{T})}
\left(\big(\Omega_{\Lambda}^{i}\big((1-e)\Lambda_0\big)(i)\big) e, e \Lambda e(-1)\right),$$
then there were a morphism
$$
0 \neq \tilde{s} \in \mathrm{Hom}^{0}_{D^{b}(\mathrm{grmod}\, R^{T})}\big(P^{-i}(i)e, e \Lambda e(-1)\big),
$$
where $\tilde{s}$ is the composition of $s$
with the morphism $P^{-i} e(i) \twoheadrightarrow 
\big(\Omega_{\Lambda}^{i}\big( (1-e)\Lambda_0\big)(i) \big)e$ 
in $P^{\bullet}_{-i}$.
Since $P^{-i}e(i)$ is the direct sum of some summands of 
$\Lambda e (-1)$ and $ \Lambda e(0)$,
there is a summand $e \Lambda e (-1)$ of $P^{-i}e(i)$ such that the embedding
$$
e \Lambda e (-1) \hookrightarrow \Lambda e (-1) \hookrightarrow P^{-i}(i)e
$$
composed with $\tilde{s}$ is the identity.
Now fix the above summand $e \Lambda (-1)$ of $P^{-i}(i)$.
By Theorem \ref{Sing} and Lemma \ref{Depth2}, we have
\begin{align*}
&\mathrm{Hom}_{\mathrm{grmod}\, R^{T}}
\left(\big(\Omega_{\Lambda}^{i} \big((1-e)\Lambda_0\big)
(i)\big) e, e \Lambda e(-1)\right)\\
 & \cong
\mathrm{Hom}_{\mathrm{tails}\, R^{T}}\left(\big(\Omega_{\Lambda}^{i}\big((1-e)\Lambda_0\big)
(i) \big)e, e \Lambda e(-1)\right) \\
&\cong  \mathrm{Hom}_{\mathrm{tails}\, \Lambda}\left(
\Omega_{\Lambda}^{i} \big((1-e)\Lambda_0\big)
(i), e \Lambda(-1)\right)\\
&\cong  \mathrm{Hom}_{\mathrm{grmod}\, \Lambda}\big(\Omega_{\Lambda}^{i} \big((1-e)\Lambda_0\big)
(i), e \Lambda(-1)\big).
\end{align*}
Denote the above isomorphism by
$$
\begin{array}{cl}
\kappa:& \mathrm{Hom}_{\mathrm{grmod}\, \Lambda}
\big(\Omega_{\Lambda}^{i} \big((1-e)\Lambda_0\big)(i), e \Lambda(-1)\big)\\
&\stackrel{\cong}\rightarrow
\mathrm{Hom}_{\mathrm{grmod}\, R^{T}}\left(
\big(\Omega_{\Lambda}^{i} \big((1-e)\Lambda_0\big)(i)\big)
 e,
e \Lambda e(-1)\right).
\end{array}
$$
By the same method, we also have the following two isomorphisms:
$$
\kappa': \mathrm{Hom}_{\mathrm{grmod}\, \Lambda}\big(e \Lambda(-1), e \Lambda(-1)\big)
\stackrel{\cong}\rightarrow
\mathrm{Hom}_{\mathrm{grmod}\, R^{T}}\big(e\Lambda e(-1), e \Lambda e(-1)\big)
$$
and
\begin{equation*}
\begin{array}{cl}
\kappa'': &\mathrm{Hom}_{\mathrm{grmod}\, \Lambda}\big(e\Lambda(-1),
\Omega_{\Lambda}^{i} \big((1-e)\Lambda_0\big)(i)\big)\\
&\stackrel{\cong}\rightarrow \mathrm{Hom}_{\mathrm{grmod}\, R^{T}}\left(e \Lambda e(-1),
\big(\Omega_{\Lambda}^{i}\big((1-e)\Lambda_0\big)(i) \big) e\right).
\end{array}
\end{equation*}
Denote by $p_{i}$ the composition of morphisms
$e \Lambda(-1) \hookrightarrow P^{-i}(i) \twoheadrightarrow
\Omega_{\Lambda}^{i} \big((1-e)\Lambda_0\big)(i)$ in $\mathrm{grmod}\,\Lambda$.
Then we have
$$
\kappa'(s \circ p_{i}) = \kappa(s) \circ \kappa''(p_{i}) = id_{e \Lambda e(-1)}.
$$
Since $s \circ p_{i} \in \mathrm{Hom}_{\mathrm{grmod}\, \Lambda}\big(e\Lambda(-1), e \Lambda(-1)\big) = e \Lambda e$,
we have that $s \circ p_{i}$ is equal to $id_{\Lambda}$ up to a scalar.
Thus $e \Lambda(-1)$ is a summand of $\Omega_{\Lambda}^{i} \big((1-e)\Lambda_0\big)(i)$ and $p_{i}$
is injective.
However, this contradicts to the construction of the resolution $P^{\bullet}$ and the fact that
$$
\mathrm{Hom}^{j}_{D^{b}(\mathrm{grmod}\, \Lambda)}\big(P^{\bullet}, e \Lambda(i)\big) = 0
$$
for any $i, j \in \mathbb{Z}$.

Therefore we have
\[
\mathrm{Hom}^{0}_{D^{b}(\mathrm{grmod}\, R^{T})}\left(
L_{e\Lambda e}\big(\big(\Omega_{\Lambda}^{i}
\big( (1-e)\Lambda_0\big)(i) \big)e\big), e \Lambda e(-1)\right) = 0.
\qedhere\]
\end{proof}



\subsection{Ext-groups of the generator}


We first recall the following two lemmas due to Minamoto and Mori and
to Mori and Ueyama respectively.

\begin{lemma}[{\cite[Proposition 4.4]{MM}}]\label{VaniExt}
Let $\Lambda$ be an AS-regular algebra of dimension $d \geq 1$ with Gorenstein parameter $a$.
Then we have
$$
\mathrm{Hom}^{q}_{D^{b}(\mathrm{tails}\, \Lambda)}\big(\Lambda(i), \Lambda(j + ma)\big) = 0
$$
for $q \neq 0, 0\leq i, j \leq a -1$ and $m \geq 1$.
\end{lemma}

\begin{lemma}[\cite{MU2}]\label{Torsion}
Let $A$ be a Noetherian AS-regular algebra, and $e$ be an idempotent of $A$ such that
$eAe$ is a Noetherian graded algebra. If
$
(-)\otimes_A Ae: \mathrm{tails}\, A \rightarrow \mathrm{tails}\, eAe
$
is an equivalence functor, then so is
$(-)\otimes_A Ae: D^{b}(\mathrm{tails}\,A) \rightarrow D^{b}(\mathrm{tails}\, eAe)$.
\end{lemma}

With these two lemmas, we are able to show the following.

\begin{lemma}\label{VaniExt1}
Let $E_Q$ be given by Definition \ref{def:EQ}.
For $r<-2$, we have $$
\mathrm{Hom}^{r}_{D_{sg}^{gr}(R^{T})}(E_{Q}, E_{Q}) = 0.
$$
\end{lemma}

\begin{proof}
We have the following isomorphisms
\begin{eqnarray*}
 &&\mathrm{Hom}^{r}_{D_{sg}^{gr}(R^{T})}(E_{Q}, E_{Q})\\
 && \cong
\mathrm{Hom}^{r}_{D^{b}(\mathrm{tails}\, R^{T})}(\Phi E_{Q}, \Phi E_{Q}) \\
 &&\cong  \mathrm{End}^{r}_{D^{b}(\mathrm{tails}\, R^{T})}\left(\pi \left(L_{e \Lambda e}
 \left( \Big(\bigoplus^{n-1}_{i = 1}\big(\Omega_{\Lambda}^{i}\big((1-e)\Lambda_0\big)(i)\big) e\Big)
 \oplus (1 - e)\Lambda e\right)\right)\right)\\
 &&\cong  \mathrm{End}^{r}_{D^{b}(\mathrm{tails}\, \Lambda)}
 \left(\pi\left(L_{e \Lambda}\left(\Big(\bigoplus^{n-1}_{i = 1}
 \Omega_{\Lambda}^{i}\big((1-e)\Lambda_0\big)(i) \Big)\oplus (1 - e)\Lambda\right)\right)\right), 
\end{eqnarray*}
where 
the second isomorphism holds by
Lemma \ref{lemma:Tocom},
and the third isomorphism holds by
Lemma \ref{Torsion}.
Since $\Lambda$ is a
graded locally finite-dimensional algebra such that $\Lambda_{0}$ is a direct sum of $k$,
by \cite{KS} again (see Definition \ref{def:minimalapprox}
and the paragraph following it) the minimal left $\mathrm{add}(e \Lambda)$-approximation
exists and is unique up to isomorphism for any graded $\Lambda$-module.

Recall from the previous subsection that
$P^\bullet$ is a  bounded graded projective $\Lambda$-module resolution
of $(1-e)\Lambda_0$ (see \eqref{def:Presol}).
Since
$$
\mathrm{Hom}^{j}_{D^{b}(\mathrm{grmod}\, \Lambda)}(P^{\bullet}_{-i}(i), e \Lambda) =
\mathrm{Hom}^{j+i}_{D^{b}(\mathrm{grmod}\, \Lambda)}(P^{\bullet}(i), e \Lambda) = 0
$$
for any $j \geq 1$, there is a morphism $q_{-i}:  \Omega_{\Lambda}^{i}\big((1-e)\Lambda_0\big)
(i) \rightarrow e \Lambda^{\oplus m_{i}}$
for some $m_{i} \in\mathbb{N}$ such that $L_{e \Lambda}\left(
\Omega_{\Lambda}^{i}\big((1-e)\Lambda_0\big)
(i)\right)$
is equal to
$$
0 \rightarrow (1-e)\Lambda (-n) \rightarrow P^{1-n}(-n+1)
\rightarrow \cdots \rightarrow \Omega_{\Lambda}^{i}\big((1-e)\Lambda_0\big) 
\xrightarrow{q_{-i}} e \Lambda^{\oplus m_{i}} \rightarrow 0
$$
in $D^{b}(\mathrm{grmod}\, \Lambda)$, which we denote by $\widetilde{P^{\bullet}_{-i}}$.

By Lemmas \ref{Depth2} and \ref{VaniExt}, we have 
$$
\mathrm{End}^{r}_{D^{b}(\mathrm{tails}\,\Lambda)}\Big(\pi\big(\widetilde{P^{\bullet}_{-i}(i)}\big) \Big)=
\mathrm{End}^{r}_{D^{b}(\mathrm{grmod}\, \Lambda)}\Big(\widetilde{P^{\bullet}_{-i}(i)}\Big)
$$
for $n-1 \geq i \geq 2$ and therefore
\begin{align*}
&\mathrm{End}^{r}_{D^{b}(\mathrm{tails}\, \Lambda)}\left(\pi\left(\bigoplus^{n-1}_{i = 2}
 \widetilde{P^{\bullet}_{-i}(i)} \oplus \bigoplus^{1}_{i = 0}(1 - e)\Lambda (i)\right)\right)\\
&=\mathrm{End}^{r}_{D^{b}(\mathrm{grmod}\, \Lambda)}\left(\bigoplus^{n-1}_{i = 2}
\widetilde{P^{\bullet}_{-i}(i)} \oplus \bigoplus^{1}_{i = 0}(1 - e)\Lambda (i)\right).
\end{align*}
Since $\Omega_{\Lambda}^{1}\big((1-e)\Lambda_0\big)
(1) \cong (1 - e)\Lambda (1)$ in $\mathrm{tails}\,\Lambda$,
the above equality implies that
\begin{align*}
&\mathrm{End}^{r}_{D^{b}(\mathrm{tails}\,\Lambda)}\left(\pi \left(L_{e \Lambda}
\left(\Big(\bigoplus^{n-1}_{i = 1}\Omega_{\Lambda}^{i}\big((1-e)\Lambda_0\big)(i)\Big) 
\oplus (1 - e)\Lambda\right)\right)\right)\\
& =
\mathrm{End}^{r}_{D^{b}(\mathrm{grmod}\, \Lambda)}\left( \bigoplus^{n-1}_{i = 2}
\widetilde{P^{\bullet}_{-i}(i)} \oplus \bigoplus^{1}_{i = 0}(1 - e)\Lambda (i)\right).
\end{align*}
Since
the homologies of the complex $\bigoplus^{n-1}_{i = 2}
\widetilde{P^{\bullet}_{-i}(i)} \oplus \bigoplus^{1}_{i = 0}(1 - e)\Lambda (i)$
are concentrated in degree $0$ and $1$,
we thus get
$$
\mathrm{End}^{r}_{D^{b}(\mathrm{grmod}\, \Lambda)}\left(\bigoplus^{n-1}_{i = 2}
\widetilde{P^{\bullet}_{-i}(i)} \oplus \bigoplus^{1}_{i = 0}(1 - e)\Lambda (i)\right) = 0
$$
for $r \leq -2$.
\end{proof}

\begin{lemma}\label{VaniExt2}
Let $E_Q$ be given by Definition \ref{def:EQ}.
For $r\ge 1$, we have $$
\mathrm{Hom}^{r}_{D_{sg}^{gr}(R^{T})}(E_{Q}, E_{Q}) = 0.
$$
\end{lemma}

\begin{proof}
We consider the following vector space
$$
\mathrm{Hom}^{0}_{D^{b}(\mathrm{grmod}\, \Lambda)}\left(\bigoplus^{n-1}_{i = 2}
\widetilde{P^{\bullet}_{-i}(i)} \oplus \bigoplus^{1}_{i = 0}(1 - e)\Lambda (i), \bigoplus^{n-1}_{i = 2}
\widetilde{P^{\bullet}_{-i}(i)} \oplus \bigoplus^{1}_{i = 0}(1 - e)\Lambda (i)[r]\right).
$$
By the fact that
$$
\mathrm{Hom}^{r}_{D^{b}(\mathrm{grmod}\, \Lambda)}(1-e, \Lambda(j)) = 0
$$
for either  $r \neq n-1$ and $j \in \mathbb{Z}$, or $r=n-1$ and $j\ne -n$,
and by the first condition in Definition \ref{def:minimalapprox}, we have that
these vector spaces
$$
\mathrm{Hom}^{0}_{D^{b}(\mathrm{grmod}\, \Lambda)}\left(\bigoplus^{n-1}_{i = 2}
\widetilde{P^{\bullet}_{-i}(i)} \oplus \bigoplus^{1}_{i = 0}(1 - e)\Lambda (i), \bigoplus^{n-1}_{i = 2}
\widetilde{P^{\bullet}_{-i}(i)} \oplus \bigoplus^{1}_{i = 0}(1 - e)\Lambda (i)[r]\right)
$$
are all trivial.
\end{proof}

\begin{lemma}\label{VaniExt3}
Let $E_Q$ be given by Definition \ref{def:EQ}.
For $r=-1$, we have $$
\mathrm{Hom}^{r}_{D_{sg}^{gr}(R^{T})}(E_{Q}, E_{Q}) = 0.
$$
\end{lemma}
\begin{proof}
We need to show
$$
\mathrm{Hom}^{-1}_{D^{b}(\mathrm{grmod}\, \Lambda)}\left(\bigoplus^{n-1}_{i = 2}
\widetilde{P^{\bullet}_{-i}(i)} \oplus \bigoplus^{1}_{i = 0}(1 - e)\Lambda (i), \bigoplus^{n-1}_{i = 2}
\widetilde{P^{\bullet}_{-i}(i)} \oplus \bigoplus^{1}_{i = 0}(1 - e)\Lambda (i)[r]\right)
$$
is trivial.

First, since the homologies of the complex $\bigoplus^{n-1}_{i = 2} \widetilde{P^{\bullet}_{-i}(i)}
\oplus \bigoplus^{1}_{i = 0}(1 - e)\Lambda (i)$ are
concentrated in degree $0$ and $1$, we have
$$
\mathrm{Hom}^{-1}_{D^{b}(\mathrm{grmod}\, \Lambda)}\left(\bigoplus^{1}_{i = 0}(1 - e)
\Lambda (i), \bigoplus^{n-1}_{i = 2} \widetilde{P^{\bullet}_{-i}(i)} 
\oplus \bigoplus^{1}_{i = 0}(1 - e)\Lambda (i)\right) = 0.
$$
Second, consider the vector spaces
\[
\mathrm{Hom}^{-1}_{D^{b}(\mathrm{grmod}\, \Lambda)}\big(\widetilde{P^{\bullet}_{-i}(i)}, \widetilde{P^{\bullet}_{-j}(j)}\big).
\]
Observe that they are isomorphic to the vector space of the following morphisms of chain complexes
\begin{equation*}\label{diag:sencond}
\xymatrixcolsep{2.5pc}
\xymatrixrowsep{1.5pc}
\xymatrix{
0 \ar[d]  \ar[r]
& \Omega^{i}_{\Lambda}\big((1-e)\Lambda_0\big)(i) 
\ar[r]^-{q_{-i}} \ar[d] & e \Lambda^{\oplus m_{i}} \ar[r] \ar[d]^{f} & 0 \ar[d] \\
0 \ar[r] & 0 \ar[r] & \Omega^{j}_{\Lambda}\big((1-e)\Lambda_0\big)(j) 
\ar[r]^-{q_{-j}} & e \Lambda }
\end{equation*}
in the bounded chain complex
category $C^{b}(\mathrm{grmod}\, \Lambda)$, for any $i, j =2,\cdots, n-1$.
Furthermore, the above diagram is contained in the following commutative diagram of chain complexes
\begin{equation*}\label{diag:third}
\xymatrixcolsep{2.5pc}
\xymatrixrowsep{1.5pc}
\xymatrix{
0 \ar[d]  \ar[r]
& P^{-i}(i) \ar[r]^-{\hat{q}_{-i}} \ar[d] & e \Lambda^{\oplus m_{i}} \ar[r] \ar[d]^{g} & 0 \ar[d] \\
0 \ar[r] & 0 \ar[r] & P^{-j+1}(j) \ar[r] & 0 }
\end{equation*}
in $C^{b}(\mathrm{grmod}\, \Lambda)$, for any $i, j=2,\cdots, n-1$,
where $\hat{q}_{-i}$ is given by
composing $q_{-i}$ with
$$
P^{-i}(i) \twoheadrightarrow \Omega^{i}_{\Lambda}\big((1-e)\Lambda_0\big)(i),
$$
and $g$ is given by composing $f$ with
$$
\Omega^{j}_{\Lambda}\big((1-e)\Lambda_0\big)(j) \hookrightarrow P^{-j+1}(j).
$$
By the definition
of $\Lambda$, the idempotent element $e$ corresponds to the $R^{T}$-module
$M^{G}_{R}(V_l)$, where $l$ is the minimal integer in $\mathcal{L}$ in Proposition \ref{SVT}.

Without loss of generality, we choose two indecomposable summands $e' \Lambda (-1) \in P^{-i}(i)$
and $e{''} \Lambda (1) \in P^{-j+1}(j)$ respectively.
By the argument after Lemma \ref{Equi}, we have that any morphism from
$e \Lambda ^{\oplus m_{i}}$ to $e{''} \Lambda (1)$
is given by some linear combination of $\{x_{j_{1}}, x_{j_2}, \cdots, x_{j_{s}}\}$
which is the subset of $\{x_{1}, x_{2}, \cdots, x_{n}\}$ consisting of elements whose weights are all negative.
Similarly,
any morphism from $e' \Lambda (-1)$ to $e \Lambda^{\oplus m_{i}}$
is given by some linear combinations of $\{x_{i_{1}}, x_{i_2}, \cdots, x_{i_{t}}\}$
which is the subset of $\{x_{1}, x_{2}, \cdots, x_{n}\}$ consisting of elements whose weights are all positive.

Thus we can write the morphism
\begin{equation*}
e \Lambda^{\oplus m_{i}} \stackrel{g}\rightarrow
P^{-j+1}(j) \twoheadrightarrow e'' \Lambda (1)
\end{equation*}
as
$
\big(\sum^{s}_{r=1} a^{1}_{r}x_{j_{r}}, \cdots, \sum^{s}_{r=1} a^{m_i}_{r}x_{j_{r}}\big),
$
where $x_{j_{r}} \in \{x_{j_{1}}, x_{j_2}, \cdots, x_{j_{s}}\}$ and
$a^{1}_{r},\cdots, a_r^{m_i} \in k$.
In the same way,
we can write the morphism
\begin{equation*}
  e{''} \Lambda (-1)\hookrightarrow
  P^{-i}(i) \stackrel{\hat{q}_{-i}}\longrightarrow e \Lambda^{\oplus m_{i}}
\end{equation*}
as
$
\big(\sum^{t}_{r=1} b^{1}_{r}x_{i_{r}}, \cdots, \sum^{t}_{r=1} b^{m_{i}}_{r}x_{i_{r}}\big),
$
where $x_{i_{r}} \in \{x_{i_{1}}, x_{i_2}, \cdots, x_{i_{t}}\}$ and $b^{1}_{r},\cdots, b_r^{m_i} \in k$.
Therefore we have
$$
\sum^{m_{i}}_{u=1}\left(\Big(\sum^{s}_{r=1} a^{u}_{r}x_{j_{r}}\Big)\Big(\sum^{t}_{r=1} b^{u}_{r}x_{i_{r}}\Big)\right) = 0
$$
in $R$. The above equality can be written in the form
$$
\sum^{j_s}_{r=1} x_{j_r} f_{j_r} = 0,
$$
where $f_{j_r}$ is a linear combination of $\{x_{i_{1}}, x_{i_2}, \cdots, x_{i_{t}}\}$ for any $j_r$.
However, since $x_{j_1} \nmid \sum^{j_s}_{r=2} x_{j_r} f_{j_r}$,
the above equality is impossible unless all $f_{j_r}=0$.
In fact, we have
$$
x_{j_1} \notin (x_{j_2}, \cdots, x_{j_s}, x_{i_{1}}, x_{i_2}, \cdots, x_{i_{t}}).
$$
Thus we obtain that
\[
\mathrm{Hom}^{-1}_{D^{b}(\mathrm{grmod}\, \Lambda)}\big(\widetilde{P^{\bullet}_{-i}(i)}, \widetilde{P^{\bullet}_{-j}(j)}\big)=0.
\]
In the completely analogous way, we have 
$$
\mathrm{Hom}^{-1}_{D^{b}(\mathrm{grmod}\,\Lambda)}\Big(\bigoplus^{n-1}_{i = 2}
\widetilde{P^{\bullet}_{-i}(i)}, \bigoplus^{1}_{i = 0}(1 - e)\Lambda (i)\Big) = 0,
$$
and therefore we have
\[
\mathrm{Hom}^{-1}_{D^{b}(\mathrm{grmod}\,\Lambda)}\Big(\bigoplus^{n-1}_{i = 2}
\widetilde{P^{\bullet}_{-i}(i)} \oplus \bigoplus^{1}_{i = 0}(1 - e)\Lambda (i),
\bigoplus^{n-1}_{i = 2} \widetilde{P^{\bullet}_{-i}(i)} \oplus \bigoplus^{1}_{i = 0}(1 - e)\Lambda (i)\Big) = 0.
\qedhere\]
\end{proof}

Combining the above Lemmas \ref{VaniExt1},
\ref{VaniExt2} and \ref{VaniExt3} we get the following.

\begin{proposition}\label{ExtVani}
Let $E_Q$ be given by Definition \ref{def:EQ}.
Then we have
$
\mathrm{Hom}^{r}_{D_{sg}^{gr}(R^{T})}(E_{Q}, E_{Q}) = 0
$
for any $r \neq 0$.
\end{proposition}

\begin{proof}[Proof of Theorem \ref{MainTh} (when $G=T$)]
By Propositions \ref{fully} and \ref{ExtVani},
$E_{Q}$ is a
tilting object in $D^{gr}_{sg}(R^{T})$, from which
the theorem follows.
\end{proof}

\section{Proof of Theorem \ref{MainTh}: the second case}\label{sect:proofofmain2}

Now we investigate the case when $G$ is the product of the one-dimensional
torus with a finite abelian group.
By \cite[Lemma 3.15]{SV}, an NCCR  can be constructed for the case where $G$
is not a connected group if
an NCCR exists for the connected component $G_{0}$ containing the identity
element.
Based on this lemma, we construct an NCCR of $R^{G}$ as follows.
Let
$$
\Lambda : = \mathrm{End}_{\mathrm{mod} \,R^{T}}\Big(\bigoplus_{\chi \in \mathcal{L}} M^{T}_{R}(V_{\chi})\Big)
\cong \mathrm{End}_{\mathrm{mod}(R, T)}\Big(\bigoplus_{\chi \in \mathcal{L}} R \otimes V_{\chi}\Big)
$$
be the $\mathrm{NCCR}$ of $R^{T}$ constructed in \S3,
where $T=k^\times \subset G$ is the maximal one-dimensional torus containing the identity
element in $G$. Let $H: = G/T$, which is a finite abelian group in $\mathrm{SL}(N, k)$.
Set
$$
\Lambda' : = \mathrm{End}_{\mathrm{mod} \,R^{G}}\Big(\bigoplus_{\chi \in \mathcal{L}} M^{G}_{R}(U_{\chi})\Big)
\cong \mathrm{End}_{\mathrm{mod}(R, G)}\Big(\bigoplus_{\chi \in \mathcal{L}}R \otimes U_{\chi}\Big),
$$
which is an $\mathrm{NCCR}$ of $R^{G}$, where
$U_{\chi} : = \mathrm{Ind}_{T}^{G}(V_{\chi})$. We have
$$
\Lambda' \cong M^{G}_{R}\Big(\mathrm{End}_{\mathrm{mod}(k,G)}\big(\mathrm{Ind}_{T}^{G}
(\bigoplus_{\chi \in \mathcal{L}} V_{\chi})\big)\Big).
$$
For any $\chi, \chi' \in \mathcal{L}$, we have
\begin{align*}
\mathrm{Hom}_{\mathrm{mod}(k, G)}(U_{\chi}, U_{\chi'}) & \cong
 \mathrm{Hom}_{\mathrm{mod}(k, G)}
 \big(\mathrm{Ind}_{T}^{G}(V_{\chi}), \mathrm{Ind}_{T}^{G}(V_{\chi'})\big)\\
 &\cong  \mathrm{Hom}_{\mathrm{mod}(k, T)}
 \big(\mathrm{Ind}_{T}^{G}(V_{\chi}), V_{\chi'}\big),
\end{align*}
where $\mathrm{Ind}_{T}^{G}(V_{\chi})$ is
considered as a $T$-representation in the last term of above equality.
At the same time, the $T$ representation $\mathrm{Ind}_{T}^{G}(V_{\chi})$ is
isomorphic to $kH \otimes V_{\chi}$, whose $T$-action is induced from $V_{\chi}$.
Thus we have
$$
\Lambda' \cong kH \otimes \Lambda,
$$
and its multiplicative structure is given by those of $kH$ and $\Lambda$.

Next, we endow $\Lambda'$ with a canonical grading from $R$
such that $\mathrm{deg}(h \otimes 1) = 0$ for any $h \in H$.
Then we obtain the following two
propositions and a lemma, whose proofs are completely analogous
to the case when $G$ is the one-dimensional torus, and hence are omitted:

\begin{proposition}[Compare with Proposition \ref{ToGo}]\label{ToGo2}
$R^{G}$ is a Noetherian graded Gorenstein algebra of dimension $n-1$
with Gorenstein parameter $n$.
\end{proposition}

\begin{theorem}[Compare with Theorem \ref{ToAS}]\label{ToAS2}
$\Lambda'$ is an AS-regular algebra of dimension $n-1$ with Gorenstein parameter $n$.
\end{theorem}

\begin{theorem}[Compare with Theorem \ref{Sing}]\label{Sing2}
Let $R = k[x_{1}, x_{2}, \cdots, x_{n}]$, and $T$ be the maximal one-dimensional torus in $G$
such that $T$ contains the identity element of $G$, which acts
on $R$ with weights $\chi = (\chi_{1}, \chi_{2}, \cdots, \chi_{n})$.
If the action of $G$ on $V$ is generic and unimodular,
then the following are equivalent:
\begin{enumerate}
\item[$(1)$] $\mathrm{gcd}(\chi_i,\chi_j)=1$ whenever
$\chi_i\chi_j<0$ for all $1\le i,j\le n$;

\item[$(2)$] the functor $( - ) e': \mathrm{tails}\,
\Lambda'\rightarrow \mathrm{tails}\, R^{G}$ is an equivalence functor
with inverse $( - ) \otimes_{e' \Lambda e'} e' \Lambda$,
where $e' :=\left(\frac{1}{|H|} \sum_{h\in H} h\right) \otimes e$
and $e \in \Lambda$ is the idempotent corresponding to $R^T$,
which is a direct summand of
$\bigoplus_{\chi \in \mathcal{L}} M^{T}_{R}(V_{\chi})$;

\item[$(3)$] $R^{G}$ a is graded isolated singularity;

\item[$(4)$] $\mathrm{Spec}(R^{G})$ has a unique isolated singularity at the origin.
\end{enumerate}
\end{theorem}

By Proposition \ref{ToAS2}, we have an object
$$
E'_{Q} = \nu\left( L_{e' \Lambda' e'}\left(\Big(\bigoplus^{n-1}_{i = 1}
\big(\Omega_{\Lambda'}^{i} (1-e')(i)\big) e'\Big) 
\oplus (1 - e')\Lambda' e'\right)\right)
$$
in $D_{sg}^{gr}(R^{G})$.

By Proposition \ref{ToGo2} and Theorems \ref{ToAS2} and \ref{Sing2},
we obtain the following two propositions,
whose proofs are again analogous and hence are omitted:

\begin{proposition}[Compare with Proposition \ref{fully}]\label{fully2}
Under the assumptions of Theorem \ref{MainTh}, in the second case, we have
$
D_{sg}^{gr}(R^{G}) = \mathrm{thick}(E'_{Q})
$.
In other words, $E'_Q$
generates $D_{sg}^{gr}(R^{G})$ itself.
\end{proposition}

\begin{proposition}[Compare with Proposition \ref{ExtVani}]\label{ExtVani2}
Under the assumptions of Theorem \ref{MainTh}, in the second case, we have
$
\mathrm{Hom}^{r}_{D_{sg}^{gr}(R^{G})}(E'_{Q}, E'_{Q}) = 0
$
for any $r \neq 0$.
\end{proposition}

\begin{proof}[Proof of Theorem \ref{MainTh} continued
(when $G$ is the product of $k^\times$
with a finite group)]
The proof follows from the combination of
Propositions \ref{fully2} and \ref{ExtVani2}.
\end{proof}

\begin{remark}\label{rem:highdim}
The above method of constructing the tilting object
does not work for the case when $G$ contains a two
or higher dimensional torus.

In fact, our method uses \cite[Lemma 1.19]{SV}, which assumes that
the action of the connected subgroup $G_{0}$ of $G$ on $R$ is quasi-symmetric.
Now suppose $G$ contains an $m$-dimensional torus $G_0$, where $m\ge 2$,
and $R^{G_{0}}$ is an isolated singularity.
Assume the weights of the $G_0$-action on $R$ are
$v_{1}, v_{2}, \cdots, v_{n}$, each of which is in $\mathbb{Z}^{m}$.
Similarly to Theorem \ref{Sing}, the isolatedness of the singularity of $R^{G_{0}}$ implies that,
if there were a sequence of natural numbers $c_{1}, c_{2}, \cdots, c_{r}$ and a subset
$\{u_{1}, u_{2}, \cdots, u_{r}\} \subseteq\{v_{1}, v_{2}, \cdots, v_{n}\}\subset X(T)$ such that
$$
\sum_{i}c_{i}u_{i} = 0,
$$
then $\{u_{1}, u_{2}, \cdots, u_{r}\}$ spans $\mathbb{Z}^{m}$ as a $\mathbb{Z}$-module.

Since the action of $G_{0}$ on $R$ is quasi-symmetric, there is a line
$\ell \subset \mathbb{R}^{m}$
through the origin such that
$$
\{\beta_{1}, \beta_{2}, \cdots, \beta_{t}\} :=\{v_{1}, v_{2}, \cdots, v_{n}\} \bigcap \ell \neq  \emptyset
$$
and $\sum_{\beta_{i} \in \ell} \beta_{i} = 0$. Thus, $\{\beta_{1}, \beta_{2}, \cdots, \beta_{t}\}$
spans $\mathbb{Z}^{m}$ as a $\mathbb{Z}$-module. But this is impossible since
$\{\beta_{1}, \beta_{2}, \cdots, \beta_{t}\}$ is contained in a line which goes through the origin.
\end{remark}

\section{An example}\label{sect:example}

In this section, we study an example of a singularity category,
where the tilting object
in its singularity category is explicitly constructed.

Let $R = k[x_{1}, x_{2}, x_{3}, x_{4}]$. Assume 
$T=k^\times$ acts on $R$ with weights $(1, 1, -1, -1)$.
By Proposition \ref{SVT},  $\mathcal{L} = \{0, 1\}$,
and then $M = R^{T} \oplus M^{T}_{R}(V_{1})$ and $\Lambda = \mathrm{End}_{R^{T}}(M)$,
where $V_{1}$ is the one dimensional representation of $T$ with weight $1$.
Moreover, $\Lambda$ can be written as a quiver algebra $kQ_{\Lambda}/ I$, where
$Q_{\Lambda}$ is given as follows:
\begin{displaymath}
\xymatrix{
\bullet_{1} \ar@/^0.3cm/[rr]^{\bar{x}_{3}} \ar@/^0.7cm/[rr]^{\bar{x}_{4}}
&& \bullet_{2} \ar@/^0.3cm/[ll]_{\bar{x}_{1}} \ar@/^0.7cm/[ll]_{\bar{x}_{2}}\\
}
\end{displaymath}
where
vertex 1 corresponds to the summand $R^{T}$ of $M$,
vertex 2 corresponds to the summand $M^{T}_{R}(V_{1})$ of $M$,
the arrows $\bar{x}_{3}$, $\bar{x}_{4}$ correspond to the morphisms
$$
x_{3}, x_{4} \in \mathrm{Hom}_{\mathrm{mod} \,R^{T}}\big(R^{T}, M^{T}_{R}(V_{1})\big)
\cong \mathrm{Hom}_{\mathrm{mod}(k, T)}(V_{0}, R \otimes V_{1}) \subseteq R
$$
respectively, the arrows $\bar{x}_{1}$, $\bar{x}_{2}$ correspond to the morphisms
$$
x_{1}, x_{2} \in  \mathrm{Hom}_{\mathrm{mod} \,R^{T}}\big(M^{T}_{R}(V_{1}), R^{T}\big)
\cong \mathrm{Hom}_{\mathrm{mod}(k, T)}(V_{1}, R \otimes V_{0}) \subseteq R
$$
respectively, and $I$ is generated by
$$\bar{x}_{1}\bar{x}_{3}\bar{x}_{2}- \bar{x}_{2}\bar{x}_{3}\bar{x}_{1},
\quad\bar{x}_{1}\bar{x}_{4}\bar{x}_{2} - \bar{x}_{2}\bar{x}_{4}\bar{x}_{1},
\quad\bar{x}_{3}\bar{x}_{1}\bar{x}_{4}-\bar{x}_{4}\bar{x}_{1}\bar{x}_{3},
\quad
\bar{x}_{3}\bar{x}_{2}\bar{x}_{4}- \bar{x}_{4}\bar{x}_{2}\bar{x}_{3}.
$$

Now, we have a projective resolution of $\Lambda_{0} \cong k^{2}$ as
a graded $\Lambda$-module as follows:
\begin{equation}\label{eq:resln}
\xymatrix{
0 \ar[r] & \Lambda(-4) \ar[r]^-{\varphi_{3}}
& \Lambda^{\oplus 4}(-3) \ar[r]^-{\varphi_{2}} &
\Lambda^{\oplus 4}(-1) \ar[r]^-{\varphi_{1}}
&  \Lambda \ar[r]^-{\varphi_{0}} & \Lambda_{0} \ar[r] & 0 }
\end{equation}
where:
\begin{enumerate}
\item[(1)] $\varphi_{0}$ is given by the canonical projection
$\Lambda = \Lambda \twoheadrightarrow \Lambda_{0}$;
\item[(2)] $\varphi_{1}$ is given by
 $$
 \varphi_{1}(a, b, c, d) = (\bar{x}_{1}a + \bar{x}_{2}b + \bar{x}_{3}c + \bar{x}_{4}d),
 $$
for any $(a, b, c, d) \in  \Lambda^{\oplus 4}(-1)$;
\item[(3)] $\varphi_{2}$ is given by
 $$
 \varphi_{1}(a, b, c, d) = (\bar{x}_{4}\bar{x}_{2}a - \bar{x}_{3}\bar{x}_{2}b,  -\bar{x}_{4}\bar{x}_{1}a
 + \bar{x}_{3}\bar{x}_{1}b, \bar{x}_{1}\bar{x}_{4}c - \bar{x}_{2}\bar{x}_{4}d, -\bar{x}_{1}\bar{x}_{3}c
 + \bar{x}_{2}\bar{x}_{3}d),
 $$
for any $(a, b, c, d) \in  \Lambda^{\oplus 4}(-3)$;
\item[(4)] $\varphi_{3}$ is given by
 $$
 \varphi_{3}\big(1_{\Lambda}(-4)\big)
 = \big(\bar{x}_{3}(-4), \bar{x}_{4}(-4), \bar{x}_{2}(-4), \bar{x}_{1}(-4)\big)
 \in \Lambda^{\oplus 4}(-3),
 $$
and $1_{\Lambda}$ is the identity element of $\Lambda$.
\end{enumerate}

Moreover, for the submodule $(1-e) \Lambda_{0} \cong k$ of $\Lambda_{0}$,
we have a projective resolution of $ (1-e) \Lambda_{0}$ as
a graded $\Lambda$-module as follows:
\begin{equation*}
\begin{split}
\xymatrixrowsep{1.5pc}
\xymatrix{
0 \ar[r] & (1-e) \Lambda(-4) \ar[r]^-{\varphi_{3}\mid_{(1-e) \Lambda}}
& e \Lambda^{\oplus 2}(-3) \ar[d]^-{\varphi_{2}\mid_{e \Lambda^{\oplus 2} }}&&&\\
&&
e \Lambda^{\oplus 2}(-1) \ar[r]^-{\varphi_{1}\mid_{e \Lambda^{\oplus 2} }}
&  (1-e) \Lambda \ar[r]^-{\varphi_{0}\mid_{(1-e) \Lambda}} & (1-e) \Lambda_{0} \ar[r] & 0,
}
\end{split}
\end{equation*}
which is a subcomplex of above resolution \eqref{eq:resln}
of $\Lambda_{0}$,  where
the idempotent $e$ corresponds to the vertex $1$, the idempotent
$1-e$ corresponds to the vertex $2$, and
\begin{enumerate}
\item[(1)] $\varphi_{0}|_{(1-e) \Lambda}$ is given by the canonical projection
$\Lambda|_{(1-e) \Lambda} = (1-e)\Lambda \twoheadrightarrow (1-e)\Lambda_{0}$;

\item[(2)] $\varphi_{1}|_{e \Lambda^{\oplus 2}}$ is given by
 $$
 \varphi_{1}|_{e \Lambda^{\oplus 2}}(c, d) = ( \bar{x}_{3}c + \bar{x}_{4}d),
 $$
for any $(c, d) \in  \Lambda^{\oplus 2}(-1)$;

\item[(3)] $\varphi_{2}|_{e \Lambda^{\oplus 2}}$ is given by
 $$
 \varphi_{2}|_{e \Lambda^{\oplus 2}}(c, d)
 = (\bar{x}_{1}\bar{x}_{4}c - \bar{x}_{2}\bar{x}_{4}d, -\bar{x}_{1}\bar{x}_{3}c
 + \bar{x}_{2}\bar{x}_{3}d),
 $$
for any $(c, d) \in  \Lambda^{\oplus 2}(-3)$;

\item[(4)] $\varphi_{3}|_{(1-e) \Lambda}$ is given by
 $$
 \varphi_{3}|_{(1-e) \Lambda} \big(1_{\Lambda}(-4)\big)
 = \big(\bar{x}_{2}(-4), \bar{x}_{1}(-4)\big)
 \in \Lambda^{\oplus 2}(-3),
 $$
and $1_{\Lambda}$ is the identity element of $\Lambda$.
\end{enumerate}
By a slight abuse of notation, we again use $\varphi_{0}$, $\varphi_{1}$,
$\varphi_{2}$ and $\varphi_{3}$
to denote $\varphi_{0}|_{(1-e) \Lambda}$, $\varphi_{1}|_{e \Lambda}$,
$\varphi_{2}|_{e \Lambda}$ and
$\varphi_{3}|_{(1-e) \Lambda}$ respectively.

From the above projective resolutions,
we obtain the tilting object $E_{Q}$ of $D_{sg}^{gr}(R^{T})$,
which is the direct
sum of the following objects:
\begin{enumerate}
\item[$(1)$] $\nu\big((1-e)\Lambda e \big)$;
\item[$(2)$] $\nu\big( (1-e)\Lambda e (1)\big)$;
\item[$(3)$]  $\nu\big(
(e\Lambda e)^{\oplus 2}(1) \xrightarrow{\varphi_{1} \otimes _{\Lambda} \Lambda e}
 (1-e) \Lambda e(2)\big)$;
\item[$(4)$]   $\nu\big((1-e)\Lambda e(-1) \xrightarrow{\varphi_{3} \otimes _{\Lambda} \Lambda e}
e \Lambda e^{\oplus 2}\big)$.
\end{enumerate}
In the above four objects, (1) is obvious.
Object (2) is given as follows:
observe that the morphism space
$$
\mathrm{Hom}_{D^{b}(\mathrm{grmod}\, R^{T})}
\Big(\big(\Omega_{\Lambda}^{1}\big((1-e)\Lambda_0\big)(1) \big)e, e \Lambda e\Big)
$$
vanishes; from the definition of the cone of the left minimal 
$\mathrm{add}(e \Lambda e)$-approximation, we then have
$L_{e \Lambda e}\left( \big(\Omega_{\Lambda}^{1}\big((1-e)\Lambda_0\big)(1)
\big)  e\right)= \big(\Omega_{\Lambda}^{1}\big((1-e)\Lambda_0\big)
(1)\big)e.$
Thus, we get
$$
L_{e \Lambda e} \left(\big(\Omega_{\Lambda}^{1}\big((1-e)\Lambda_0\big)
(1) \big)e\right)
= (1-e)\mathrm{ker}(\varphi_{0})e(1)
=(1-e) \Lambda e (1),
$$
where the first equality holds since we have
$$
\big(\Omega_{\Lambda}^{1}\big((1-e)\Lambda_0\big)\big)
 e = (1-e)\mathrm{ker}(\varphi_{0})e
$$
in $D^{b}(\mathrm{tails} R^{T})$.

Object (3) comes as follows:
observe to that the morphism space
$$
\mathrm{Hom}_{D^{b}(\mathrm{grmod}\, R^{T})}
\Big(
\big(\Omega_{\Lambda}^{2}\big((1-e)\Lambda_0\big)(2)\big) e, e \Lambda e\Big)
$$
vanishes, and thus $$L_{e \Lambda e}
\left(\big(\Omega_{\Lambda}^{2}\big((1-e)\Lambda_0\big)(2) \big)e\right)
= \big(\Omega_{\Lambda}^{2}\big((1-e)\Lambda_0\big)(2)\big) e.$$
Therefore $ L_{e \Lambda e} \left(\big((\Omega_{\Lambda}^{2}
\big((1-e)\Lambda_0\big)(2)\big) e\right)$
is isomorphic to the following complex
\begin{equation*}
 (e\Lambda e)^{\oplus 2}(1) \xrightarrow{\varphi_{1} \otimes _{\Lambda} \Lambda e}
 (1-e) \Lambda e(2),
\end{equation*}
by the definition of syzygy, where $(e\Lambda e)^{\oplus 2}(1)$ is in the degree zero 
and $(1-e) \Lambda e(2)$
is in the degree one of this complex.

Object (4) is due to the following:
from the above
resolution \eqref{eq:resln} we have $\Omega_{\Lambda}^{3}\big((1-e)\Lambda_0\big)
 \cong (1 - e) \Lambda(-4)$,
and then
$\big(\Omega_{\Lambda}^{3}\big((1-e)\Lambda_0\big)
(3)\big) e \cong (1-e)\Lambda e(-1)$.
Thus we get
$$
\mathrm{Hom}_{D^{b}(\mathrm{grmod}\, R^{T})}
\Big(\big(\Omega_{\Lambda}^{3}\big((1-e)\Lambda_0\big)
(3)  \big)e, e \Lambda e\Big) \cong
\mathrm{Hom}_{D^{b}(\mathrm{grmod}\, R^{T})}\big((1-e)\Lambda e(-1),
e \Lambda e\big) \cong k^{\oplus 2}.
$$
Suppose $f_1, f_2 \in \mathrm{Hom}_{D^{b}(\mathrm{grmod}\,
R^{T})}\Big(\big(\Omega_{\Lambda}^{3}\big((1-e)\Lambda_0\big)
(3) \big)e, e \Lambda e\Big)$
are the bases of the above vector space. Then the morphism
$$
f_{1} \oplus f_{2} :\big(\Omega_{\Lambda}^{3}\big((1-e)\Lambda_0\big)
(3) \big)e
\rightarrow (e \Lambda e)^{\oplus 2}
$$
gives the minimal left $\mathrm{add}(e \Lambda e)$-approximation of
$\big(\Omega_{\Lambda}^{3}\big((1-e)\Lambda_0\big)(3)\big)
 e$.
Moreover, with the identification of 
$\big(\Omega_{\Lambda}^{3}\big((1-e)\Lambda_0\big)
(3)\big) e$
with $(1-e)\Lambda e(-1)$, we see that $f_1 \oplus f_2$
is identified with $\varphi_{3} e$.
Thus $L_{e \Lambda e}\left(\big(\Omega_{\Lambda}^{3}\big((1-e)\Lambda_0\big)
(3)\big) e\right)$ is equal to the following complex
\begin{equation*}
 (1-e)\Lambda e(-1) \xrightarrow{\varphi_{3} \otimes _{\Lambda} \Lambda e}
 e \Lambda e^{\oplus 2},
\end{equation*}
where $(1-e)\Lambda e(-1)$ is in the degree zero and $e \Lambda e^{\oplus 2}$
is in the degree one of this complex.

In fact, we can explicitly compute $\mathrm{Hom}_{D_{sg}^{gr}(R^T)}(E_Q, E_Q)
=\mathrm{End}_{D_{sg}^{gr}(R^T)}(E_Q)$.
First of all, we have the following equalities
\begin{align*}
& \mathrm{End}_{D_{sg}^{gr}(R^T)} \left( \nu\left( \bigoplus^{3}_{i = 1}
\big(\Omega_{\Lambda}^{i}\big((1-e)\Lambda_0\big)(i)
\big) e\oplus ( 1-e)\Lambda e\right)\right) \\
&= \mathrm{End}_{D_{sg}^{gr}(R^T)} \left( \nu\left(\bigoplus^{3}_{i = 1}
L_{e \Lambda e}\left(\big(\Omega_{\Lambda}^{i}\big((1-e)\Lambda_0\big)(i) \big)e
\right) \oplus ( 1-e)\Lambda e\right)\right)\\
& = \mathrm{End}_{D_{sg}^{gr}(R^T)} \left( \Phi \circ \nu\left(\bigoplus^{3}_{i = 1}
L_{e \Lambda e} \left(\big(\Omega_{\Lambda}^{i}\big((1-e)\Lambda_0\big)(i)\big) e
\right) \oplus ( 1-e)\Lambda e\right)\right) \\
& = \mathrm{End}_{D^{b}(\mathrm{tails} \,R^T)} \left( \pi\left(\bigoplus^{3}_{i = 1}
L_{e \Lambda e} \left(\big(\Omega_{\Lambda}^{i}\big((1-e)\Lambda_0\big)(i) \big)e
\right) \oplus ( 1-e)\Lambda e\right)\right),
\end{align*}
where the first equality holds by the definition of $\nu$,
the second equality holds by Lemma \ref{Orlovembed},
and the third equality holds by
Claims \ref{claim22} and \ref{claim33} and hence
Lemma \ref{Cancommu}.

In the above equalities, we should note that
$\bigoplus^{3}_{i = 1} \big(\Omega_{\Lambda}^{i}
\big((1-e)\Lambda_0\big)(i) \big)e
\oplus ( 1-e)\Lambda e$
does not satisfy the conditions of Lemma \ref{Cancommu}, and therefore
for this object
there is no equality which is
similar with the above third equality.
Thus, to compute the above morphism spaces in
$D^{b}(\mathrm{tails} \,R^T)$, we have to transfer the computations on
$$\nu\left(\Big(\bigoplus^{3}_{i = 1} \big(\Omega_{\Lambda}^{i}\big((1-e)\Lambda_0\big)
(i) \big)e\Big)
\oplus ( 1-e)\Lambda e\right)$$ to the ones on
$$\nu\left(\Big( \bigoplus^{3}_{i = 1} L_{e \Lambda e}
\left(\big(\Omega_{\Lambda}^{i}\big((1-e)\Lambda_0\big)(i) \big)e
\right)\Big)  \oplus ( 1-e)\Lambda e\right).$$

Now, we are able to compute the morphism spaces in $D^{b}(\mathrm{tails} \,R^T)$
between the direct summands of
$\pi \left(\Big( \bigoplus^{3}_{i = 1} L_{e \Lambda e}
\left(\left(\Omega_{\Lambda}^{i}\big((1-e)\Lambda_0\big)(i) \right)e\right)
\Big) \Big)\oplus ( 1-e)\Lambda e \right)$.
The computations are quite straightforward.

Taking, for example, the summand $\pi\big((1-e)\Lambda e\big)$,
we have the following equalities.
First, \begin{align*}
& \mathrm{Hom}_{D^{b}(\mathrm{tails} \,R^T)} \left( \pi\big((1-e)\Lambda e\big),  
\pi\big((1-e)\Lambda e\big)\right) \\
& = \mathrm{Hom}_{D^{b}(\mathrm{tails} \,\Lambda )} \left( \pi\big((1-e)\Lambda\big),  
\pi\big((1-e)\Lambda\big) \right) \quad (\textup{by Lemma \ref{Torsion}})\\
&= \big((1-e) \Lambda (1-e)\big)_0 \quad (\textup{by Lemma \ref{VaniExt}}) \\
& = (R^T)_0 \\
&= k.
\end{align*}
Second,
\begin{align*}
& \mathrm{Hom}_{D^{b}(\mathrm{tails} \,R^T)} \left( \pi\big((1-e)\Lambda e\big),  
\pi\Big(L_{e \Lambda e}\big(\big(\Omega_{\Lambda}^{1}\big((1-e)\Lambda_0\big)
(1)\big) e\big)\Big)\right)\\
& = \mathrm{Hom}_{D^{b}(\mathrm{tails} \,R^T)} \big( \pi\big((1-e)\Lambda e\big),  
\pi\big((1-e)\Lambda e (1)\big) \big)\\
&= \mathrm{Hom}_{D^{b}(\mathrm{tails} \,\Lambda)}
\big( \pi\big((1-e)\Lambda\big),  
\pi\big((1-e)\Lambda (1)\big) \big) \quad (\textup{by Lemma \ref{Torsion}})\\
& = (R^T)_{1} = 0.
\end{align*}
With the same method, we have
 \begin{align*}
 & \mathrm{Hom}_{D^{b}(\mathrm{tails} \,R^T)} 
 \left( 
 \pi\big((1-e)\Lambda e\big), \pi\Big
 ( L_{e \Lambda e} \big(\big(\Omega_{\Lambda}^{2}\big((1-e)\Lambda_0\big)
 (2)\big)
 e\big)\Big) \right)
 = 0
 \end{align*}
and
\begin{align*}
 & \mathrm{Hom}_{D^{b}(\mathrm{tails} \,R^T)}
 \left( \pi\big((1 - e) \Lambda e\big) ,  \pi\Big(L_{e \Lambda e} 
 \big(\big(\Omega_{\Lambda}^{3}\big((1-e)\Lambda_0\big)
 (3)\big) e\big)\Big) \right)
= 0.
\end{align*}
For the rest summands,
the computations are similar and are left to the interested readers.
In summary, only the endomorphisms of these summands are isomorphic
to $k$, and all the rest homomorphisms vanish.
We thus have
$$\mathrm{End}_{D_{sg}^{gr}(R^{T})}(E_{Q})=
\Big\{\mathrm{diag}(k_1,k_2,k_3,k_4)\Big|k_i\in k, i=1,\cdots, 4\Big\} \cong
k^{\oplus 4}.
$$

\begin{remark}
In this paper, we have assumed that the action of $T$ on $V$ is {\it generic}, which
then implies that $\mathrm{dim}(V)>3$, since the codimension of the set of unstable
points should be less than 2. Moreover, if $\mathrm{dim}(V)\le 3$,
then the weights $\chi$ of the $T$-action do not
contain at least two positive and two negative weights
(see Lemma \ref{equivalent} and Proposition \ref{prop:equivalencecond}).
However, in this case we may still consider the situation where the
$T$-action is {\it unimodular}; that is, $\sum_{i=1}^n\chi_i=0$. This is divided into
three cases:

\begin{enumerate}
\item[(1)] If $\mathrm{dim}(V) = 1$, then $R^{T} \cong k$. This is trivial.

\item[(2)] If $\mathrm{dim}(V) = 2$, then the weights $\chi = (\chi_1, \chi_2)$
has $\chi_2 =-\chi_1$.
This implies that $R^T \cong k[x_{1}x_{2}]\cong k[x] $ and then $\mathrm{Spec}(R^T)$ is smooth.

\item[(3)] If $\mathrm{dim}(V) = 3$, then there is at least one positive
and one negative component in $\chi = (\chi_1, \chi_2, \chi_3)$.
Without loss of generality, suppose that $\chi_1 > 0$ and $\chi_{2}, \chi_{3} < 0$.
Then it is straightforward to check that the ring $R^T$ is the $A_{\chi_1}$-singularity.
Moreover, $R^T$ is a graded Gorenstein ring of dimension 2 with Gorenstein parameter 3.
By the same argument as in \S\ref{sect:proofofmain1},
we have that $R^T$ satisfies the last three conditions in Theorem \ref{Sing},
and therefore obtain a tilting object of $D^{gr}_{sg}(R^T)$ by the same procedure.
\end{enumerate}
\end{remark}


\begin{thebibliography}{100}

\bibitem{CA}
Claire Amiot,
{\it Preprojective algebras, singularity categories and orthogonal decompositions},
 Algebras, quivers and representations, 1-11,
Abel Symp., 8, Springer, Heidelberg, 2013.

\bibitem{RB}
Ragnar-Olaf Buchweitz,
{\it Maximal Cohen-Macaulay modules and Tate cohomology},
with appendices by L.L. Avramov, B. Briggs, S.B. Iyengar and J.C. Letz.
Mathematical Surveys and Monographs Volume 262.
American Mathematical Society, 2021.

\bibitem{BIY}
Ragnar-Olaf Buchweitz, Osamu Iyama and Kota Yamaura,
{\it Tilting theory for Gorenstein rings in dimension one},
Forum Math. Sigma 8 (2020), Paper No. 37 pp.

\bibitem{DW}
Will Donovan and Michael Wemyss,
{\it Noncommutative deformations and flops},
Duke Math. J. 165 (2016), no. 8, 1397-1474.

\bibitem{SE}
Samuel Eilenberg,
{\it Homological dimension and syzygies}, Ann. of Math. (2) 64 (1956), 328-336.


\bibitem{HHK}Lidia A. H\"ugel, Dieter Happel and Henning Krause (Eds.),
{\it Handbook of Tilting Theory}, London Math. Soc. Lect.
Note Series 332, Cambridge Univ. Press, Cambridge, 2007.

\bibitem{IR}
Osamu Iyama and Idun Reiten,
{\it Fomin-Zelevinsky mutation and tilting modules over Calabi-Yau algebras},
Amer. J. Math. 130 (2008), no. 4, 1087-1149.

\bibitem{IT}
Osamu Iyama and Ryo Takahashi,
{\it Tilting and cluster tilting for quotient singularities},
Math. Ann. 356 (2013), no. 3, 1065-1105.

\bibitem{KS}
Henning Krause and Manuel Saorin,
{\it On minimal approximations of modules},
Contemp. Math., 229, Amer. Math. Soc.,
Providence, RI, (1998), 227-236.

\bibitem{Lu}
Domingo Luna,
{\it Slices \'etales},
Bull. Soc. Math. France 33, 81-105 (1973).

\bibitem{M}
Izuru Mori,
{\it McKay type correspondence for AS-Gorenstein algebras},
J. Lond. Math. Soc. 88 (2013) 2071-2091.

\bibitem{MM}
Hiroyuki Minamoto and Izuru Mori,
{\it The structure of AS-Gorenstein algebras},
Adv. Math. 226 (2011), no. 5, 4061-4095.

\bibitem{MU}
Izuru Mori and Kenta Ueyama,
{\it Stable categories of graded maximal Cohen-Macaulay modules over noncommutative quotient singularities},
Adv. Math. 297 (2016), 54-92.

\bibitem{MU2}
Izuru Mori and Kenta Ueyama,
{\it Ample group action on AS-regular algebras and noncommutative graded isolated singularities},
Trans. Amer. Math. Soc. 368 (2016), no. 10, 7359-7383.

\bibitem{DR}
Dmitri Orlov,
{\it Triangulated categories of singularities and D-branes in Landau-Ginzburg models},
Tr. Mat. Inst. Steklova 246 (2004), Algebr. Geom. Metody, Svyazi i Prilozh.,
240-262; translation in
Proc. Steklov Inst. Math. 2004, no. 3 (246), 227-248.

\bibitem{DR1}
Dmitri Orlov,
{\it Triangulated categories of singularities, and equivalences between Landau-Ginzburg models},
(Russian) Mat. Sb. 197 (2006), no. 12, 117-132.

\bibitem{DR2}
Dmitri Orlov,
{\it Derived categories of coherent sheaves and triangulated categories of singularities}.
Algebra, arithmetic, and geometry: in honor of Yu. I. Manin. Vol. II,
503-531, Progr. Math., 270, Birkh\"{a}user Boston, Boston, MA, 2009.

\bibitem{SV}
\v{S}pela \v{S}penko and Michel Van den Bergh,
{\it Non-commutative resolutions of quotient singularities for reductive groups},
Invent. Math. 210 (2017), no. 1, 3-67.

\bibitem{SV1}
\v{S}pela \v{S}penko and Michel Van den Bergh,
{\it Non-commutative crepant resolutions for some toric singularities I},
Int. Math. Res. Not. IMRN 2020, no. 21, 8120-8138.

\bibitem{SV2}
\v{S}pela \v{S}penko and Michel Van den Bergh,
{\it Comparing the Kirwan and Non-commutative resolutions of quotient varieties},
arXiv:1912.01689v1.


\bibitem{V}
Michel Van den Bergh,
{\it Three-dimensional flops and noncommutative rings}.
Duke Math. J. 122 (2004), no. 3, 423-455.


\bibitem{V2}
Michel Van den Bergh,
{\it Non-commutative crepant resolutions}.
The legacy of Niels Henrik Abel, 749-770, Springer, Berlin, 2004.


\bibitem{W2}
Michael Wemyss,
{\it Noncommutative resolutions}. Noncommutative algebraic geometry,
Math. Sci. Res. Inst. Publ., 64, Cambridge Univ. Press, New York, 2016, 239-306.

\end{thebibliography}
\end{document}